\documentclass[a4paper,10pt]{article}
\usepackage{amssymb,amsmath}
\input xy
\xyoption{all}
\usepackage[all]{xy}
\usepackage{amsthm,amsmath,amsfonts,amssymb,mathrsfs}

\theoremstyle{plain}
\newtheorem{theorem}{Theorem}[section]
\newtheorem{lemma}[theorem]{Lemma}
\newtheorem{proposition}[theorem]{Proposition}
\newtheorem{corollary}[theorem]{Corollary}

\theoremstyle{definition}
\newtheorem{definition}[theorem]{Definition}

\theoremstyle{remark}
\newtheorem{remark}[theorem]{Remark}

\numberwithin{equation}{section}
\numberwithin{figure}{section}

\author{K\^az\i m \.Ilhan \.IKEDA and Erol SERBEST}

\title{Fesenko reciprocity map}

\date{}
\begin{document}

\maketitle

\begin{flushright}
\textit{Dedicated to our teacher Mehpare Bilhan}
\end{flushright}
\begin{abstract}
In \cite{fesenko-2000, fesenko-2001, fesenko-2005}, Fesenko has defined 
the non-abelian local reciprocity map
for every totally-ramified arithmetically profinite ($APF$) Galois extension
of a given local field $K$ by extending the works of
Hazewinkel \cite{hazewinkel} and Neukirch-Iwasawa \cite{neukirch}.
The theory of Fesenko extends the previous non-abelian generalizations
of local class field theory given by Koch-de Shalit
\cite{koch-deshalit} and by A. Gurevich \cite{gurevich}. 
In this paper, which is
research-expository in nature, we give a detailed account of Fesenko's work, 
and include all the proofs that are skipped in 
\cite{fesenko-2000, fesenko-2001, fesenko-2005}.

\noindent
\textbf{2000 Mthematics Subject Classification} : Primary 11S37

\noindent
\textbf{Keywords} : Local fields, higher-ramification theory, 
$APF$-extensions, 
Fontaine-Wintenberger field of norms, Fesenko reciprocity map, non-abelian 
local class field theory, $p$-adic local Langlands correspondence.
\end{abstract}


In a series of very interesting papers \cite{fesenko-2000, fesenko-2001, 
fesenko-2005}, Fesenko 
has defined the non-abelian local reciprocity map
for every totally-ramified arithmetically profinite ($APF$) Galois extension
of a given local field $K$ by extending the works of
Hazewinkel \cite{hazewinkel} and Neukirch-Iwasawa \cite{neukirch}.
``Fesenko theory'' extends the previous non-abelian generalizations
of local class field theory given by Koch and de Shalit
in \cite{koch-deshalit} and by A. Gurevich in \cite{gurevich}. 

In this paper, which is research-expository in nature, we give a very detailed 
account of Fesenko's work \cite{fesenko-2000, fesenko-2001, fesenko-2005},  
thereby complementing them by including
all the proofs. Let us describe how our paper is organized: 
In the first part, we briefly review abelian local class field theory and the 
construction of the local Artin reciprocity map following Hazewinkel method
and Neukirch-Iwasawa method.  
In parts 3 and 4, following \cite{fesenko-vostokov},
\cite{fontaine-wintenberger-1, fontaine-wintenberger-2}, and 
\cite{wintenberger-1983}, we review
the theory of $APF$-extensions over $K$, and sketch the construction of
Fontaine-Wintenberger's field of norms $\mathbb X(L/K)$ attached to an
$APF$-extension $L/K$. 
In order to do so, in part 2, we briefly review the ramification theory of $K$.
Finally, in part 5, we give a detailed construction of the Fesenko
reciprocity map $\Phi_{L/K}^{(\varphi)}$ defined for any totally-ramified 
and $APF$-Galois extension $L$ over $K$ under the assumption that
$\pmb{\mu}_p(K^{sep})\subset K$, where $p=\text{char}(\kappa_K)$, and 
investigate the functorial 
and ramification-theoretic properties of the Fesenko reciprocity maps defined 
for totally-ramified and $APF$-Galois extensions over $K$. 

In a companion paper \cite{ikeda-serbest-1}, 
we shall extend the construction of 
Fesenko to \textit{any} Galois extension of $K$ (in a fixed $K^{sep}$), and
construct the non-abelian local class field theory. 
Thus, we feel that, the present paper together with \cite{fesenko-2000, 
fesenko-2001, fesenko-2005} 
should be viewed as the technical and theoretical background, an introduction, 
as well as an appendix to our companion paper \cite{ikeda-serbest-1} 
on ``generalized Fesenko theory''.
A similar theory has been announced by Laubie in \cite{laubie}, which is
an extension of the work of Koch and de Shalit \cite{koch-deshalit}. The
relationship of Laubie theory with our generalized Fesenko theory will be 
investigated in our companion paper as well.

\subsection*{Acknowledgements}
The first named author (K. I. I) would like to thank to the Institut de
Math{\'e}matiques, ``Th{\'e}orie des Groupes, Repr{\'e}sentations et
Applications'', Universit{\'e} Paris 7, Jussieu, Paris, and to
the School of Pure Mathematics of the Tata Institute of Fundamental
Research, Mumbai, for their hospitality and support, where
some parts of this work have been completed. The second named author (E. S.)
would like to thank to T\"UBITAK for a fellowship, and to the School of 
Mathematics of the University of 
Nottingham for the hospitality and support, where some parts of this work
have been introduced. Both of the authors would like to thank I. B. Fesenko
for his interest and encouragement at all stages of this work.

\subsection*{Notation}
All through this work, $K$ will denote a local field (a complete
discrete valuation field) with finite residue
field $O_K/\mathfrak p_K=:\kappa_K$ of $q_K=q=p^f$ elements with $p$ a
prime number, where $O_K$
denotes the ring of integers in $K$ with the unique maximal 
ideal $\mathfrak p_K$.
Let $\pmb{\nu}_K$ denote the corresponding normalized valuation on $K$
(normalized by $\pmb{\nu}_K(K^\times)=\mathbb Z$), and
$\widetilde{\pmb{\nu}}$ the unique extension of $\pmb{\nu}_K$ to a fixed
separable closure $K^{sep}$ of $K$. For any sub-extension $L/K$ of
$K^{sep}/K$, the normalized form of the
valuation $\widetilde{\pmb{\nu}}\mid_L$ on $L$ will be denoted by
$\pmb{\nu}_L$.
Finally, let $G_K$ denote the absolute Galois group $\text{Gal}(K^{sep}/K)$.

\section{Abelian local class field theory}
Let $K$ be a local field.
Fix a separable closure $K^{sep}$ of the local field $K$.
Let $G_K$ denote the absolute Galois group $\text{Gal}(K^{sep}/K)$ 
of $K$. By the construction of absolute Galois groups, 
$G_K$ is a pro-finite topological group with respect to the Krull
topology. Now let $G_K^{ab}$ denote the maximal abelian Hausdorff 
quotient group $G_K/G_K'$ of the topological group $G_K$, where $G_K'$
denotes the closure of the $1^{st}$-commutator subgroup 
$[G_K,G_K]$ of $G_K$.

Recall that, abelian local class field theory for the local field $K$
establishes a unique natural algebraic and topological isomorphism
\begin{equation*}
\alpha_K:\widehat{K^\times}\xrightarrow{\sim} G_K^{ab},
\end{equation*}
where the topological group $\widehat{K^\times}$ denotes the
pro-finite completion of the multiplicative group $K^\times$, satisfying
the following conditions
\begin{itemize}
\item[(1)]
Let $W_K$ denote the Weil group of $K$. Then
\begin{equation*}
\alpha_K(K^\times)=W_K^{ab};
\end{equation*}
\item[(2)]
for every abelian extension $L/K$ (always assumed to be a
sub-extension of $K^{sep}/K$, where a separable closure $K^{sep}$ 
of $K$ is fixed all through the remainder of the text), the 
surjective and continuous homomorphism
\begin{equation*}
\alpha_{L/K}:\widehat{K^\times}\xrightarrow{\alpha_K}G_K^{ab}
\xrightarrow{\text{res}_L}\text{Gal}(L/K)
\end{equation*}
satisfies
\begin{equation*}
\text{ker}(\alpha_{L/K})=N_{L/K}(\widehat{L^\times})=
\bigcap_{K\underset{\text{finite}}\subseteq F\subseteq L}
N_{F/K}(\widehat{F^\times})=:\mathcal N_L;
\end{equation*}
\item[(3)]
for each abelian extension $L/K$, the mapping
\begin{equation*}
L\mapsto\mathcal N_L
\end{equation*}
defines a bijective correspondence 
\begin{equation*}
\{L/K:\text{abelian}\}\leftrightarrow
\{\mathcal N:\mathcal N\underset{\text{closed}}\leq\widehat{K^\times}\},
\end{equation*}
which satisfies the following conditions: for every abelian extension
$L,L_1$, and $L_2$ over $K$,
\begin{itemize}
\item[(i)]
$L/K$ is a finite extension if and only if $\mathcal
N_L\underset{\text{open}}\leq\widehat{K^\times}$ 
(which is equivalent to $(\widehat{K^\times}:\mathcal N_L)<\infty$);
\item[(ii)]
$L_1\subseteq L_2\Leftrightarrow\mathcal N_{L_1}\supseteq\mathcal
N_{L_2}$;
\item[(iii)]
$\mathcal N_{L_1L_2}=\mathcal N_{L_1}\cap\mathcal N_{L_2}$;
\item[(iv)]
$\mathcal N_{L_1\cap L_2}=\mathcal N_{L_1}\mathcal N_{L_2}$.
\end{itemize}
\item[(4)]
(Ramification theory)\footnote{We shall review the higher-ramification 
subgroups $\text{Gal}(L/K)^\nu$ of $\text{Gal}(L/K)$ (in upper-numbering) 
in the next section.}. Let $L/K$ be an abelian extension.
For every integer $0\leq i\in\mathbb Z$ and for every real number 
$\nu\in (i-1,i]$,
\begin{equation*}
x\in U_K^i\mathcal
N_L\Leftrightarrow\alpha_{L/K}(x)\in\text{Gal}(L/K)^\nu ,
\end{equation*} 
where $x\in\widehat{K^\times}$. 
\item[(5)]
(Functoriality). Let $L/K$ be an abelian extension.
\begin{itemize}
\item[(i)]
For $\gamma\in\text{Aut}(K)$,
\begin{equation*}
\alpha_K(\gamma(x))=\widetilde{\gamma}\alpha_K(x)\widetilde{\gamma}^{-1},
\end{equation*}
for every $x\in\widehat{K^\times}$, where
$\widetilde{\gamma}:K^{ab}\rightarrow K^{ab}$ is any automorphism of
the field $K^{ab}$ satisfying $\widetilde{\gamma}\mid_K=\gamma$;
\item[(ii)]
under the condition $[L:K]<\infty$,
\begin{equation*}
\alpha_L(x)\mid_{K^{ab}}=\alpha_K(N_{L/K}(x)),
\end{equation*}
for every $x\in\widehat{L^\times}$;
\item[(iii)]
under the condition $[L:K]<\infty$,
\begin{equation*}
\alpha_L(x)=V_{K\rightarrow L}\left(\alpha_K(x)\right),
\end{equation*}
for every $x\in\widehat{K^\times}$, where $V_{K\rightarrow
  L}:G_K^{ab}\rightarrow G_L^{ab}$ is the group-theoretic transfer
  homomorphism (Verlagerung).
\end{itemize}
\end{itemize}
This unique algebraic and topological isomorphism 
$\alpha_K:\widehat{K^\times}\rightarrow G_K^{ab}$ is
called the \textit{local Artin reciprocity map of $K$}.

There are many constructions of the local Artin reciprocity map of $K$ 
including the \textit{cohomological} and \textit{analytical} constructions.
Now, in the remainder of this section, we shall review the 
construction of the local Artin reciprocity
map $\alpha_K:\widehat{K^\times}\rightarrow G_K^{ab}$ of $K$ following 
Hazewinkel (\cite{hazewinkel}) and Iwasawa-Neukirch (\cite{iwasawa}, 
\cite{neukirch}). 
As usual, let $K^{nr}$ denote the maximal unramified extension of
$K$. It is well-known that $K^{nr}$ is not a complete field with
respect to the valuation $\nu_{K^{nr}}$ on $K^{nr}$ induced from the valuation
$\nu_K$ of $K$. Let $\widetilde{K}$ denotes the completion of $K^{nr}$
with respect to the valuation $\nu_{K^{nr}}$ on $K^{nr}$. For a
Galois extension $L/K$, let $L^{nr}=LK^{nr}$ and
$\widetilde{L}=L\widetilde{K}$. For each $\tau\in\text{Gal}(L/K)$,
choose $\tau^*\in\text{Gal}(L^{nr}/K)$ in such a way that:
\begin{itemize}
\item[(1)]
$\tau^*\mid_L=\tau$;
\item[(2)]
$\tau^*\mid_{K^{nr}}=\varphi^n$, for some $0<n\in\mathbb Z$, where
$\varphi\in\text{Gal}(K^{nr}/K)$ denotes the (arithmetic) Frobenius 
automorphism of $K$. 
\end{itemize}
Let the fixed-field $(L^{nr})^{\tau^*}=\{x\in L^{nr}:\tau^*(x)=x\}$ of
this chosen $\tau^*\in\text{Gal}(L^{nr}/K)$ in $L^{nr}$ be denoted by 
$\Sigma_{\tau^*}$, which satisfies $[\Sigma_{\tau^*}:K]<\infty$.

The Iwasawa-Neukirch mapping
\begin{equation*}
\iota_{L/K}:\text{Gal}(L/K)\rightarrow K^\times/N_{L/K}(L^\times)
\end{equation*}
is then defined by
\begin{equation*}
\iota_{L/K}:\tau\mapsto
N_{\Sigma_{\tau^*}/K}(\pi_{\Sigma_{\tau^*}})\mod{N_{L/K}(L^\times)}, 
\end{equation*}
for every $\tau\in\text{Gal}(L/K)$, where $\pi_{\Sigma_{\tau^*}}$ 
denotes any prime element of $\Sigma_{\tau^*}$.

Suppose now that the Galois extension $L/K$ is furthermore a 
totally-ramified and finite extension. Let $\widetilde{V}(L/K)$ be the
subgroup of the unit group $U_{\widetilde{L}}=O_{\widetilde{L}}^\times$
of the ring of integers $O_{\widetilde{L}}$ of the local field
$\widetilde{L}$ defined by
\begin{equation*}
\widetilde{V}(L/K)=\left<u^{\sigma-1}:u\in U_{\widetilde{L}},
~\sigma\in\text{Gal}(L/K)\right>.
\end{equation*}
Then the homomorphism
\begin{equation*}
\theta:\text{Gal}(L/K)\rightarrow U_{\widetilde{L}}/\widetilde{V}(L/K)
\end{equation*}
defined by
\begin{equation*}
\theta:\sigma\mapsto\frac{\sigma(\pi_L)}{\pi_L}\mod{\widetilde{V}(L/K)},
\end{equation*}
for every $\sigma\in\text{Gal}(L/K)$, makes the following triangle
\begin{equation*}
\SelectTips{cm}{}\xymatrix{
{\text{Gal}(L/K)}\ar[dd]_{\text{can.}}\ar[dr]^{\theta} & \\
 & {U_{\widetilde{L}}/\widetilde{V}(L/K)} \\
{\text{Gal}(L/K)^{ab}}\ar[ur]_{\theta_o} & 
}
\end{equation*}
commutative. The quotient $U_{\widetilde{L}}/\widetilde{V}(L/K)$ 
sits in the Serre short exact sequence
\begin{equation}
\label{serre1}
1\rightarrow\text{Gal}(L/K)^{ab}\xrightarrow{\theta_o}
U_{\widetilde{L}}/\widetilde{V}(L/K)
\xrightarrow{N_{\widetilde{L}/\widetilde{K}}} U_{\widetilde{K}}\rightarrow 1.
\end{equation}
Let $V(L/K)$ denote the subgroup of the unit group $U_{L^{nr}}$ of the
ring of integers $O_{L^{nr}}$ of the maximal unramified extension
$L^{nr}$ of the local field $L$ defined by
\begin{equation*}
V(L/K)=\left<u^{\sigma-1}:u\in U_{L^{nr}},~\sigma\in\text{Gal}(L/K)\right>.
\end{equation*}  
The quotient $U_{L^{nr}}/V(L/K)$ sits in the Serre short exact sequence
\begin{equation}
\label{serre2}
1\rightarrow\text{Gal}(L/K)^{ab}\xrightarrow{\theta_o}
U_{L^{nr}}/V(L/K)\xrightarrow{N_{L^{nr}/K^{nr}}} U_{K^{nr}}\rightarrow 1.
\end{equation}
As before, let
$\varphi\in\text{Gal}(K^{nr}/K)$ denote the Frobenius automorphism of
$K$. Fix any extension of the automorphism $\varphi$ of $K^{nr}$ 
to an automorphism of $L^{nr}$, denoted again by $\varphi$. 
Now, for any $u\in U_K$, there exists $v_u\in U_{L^{nr}}$, such that
$u=N_{L^{nr}/K^{nr}}(v_u)$. 
Then, the equality
\begin{equation*}
N_{L^{nr}/K^{nr}}(\varphi(v_u))=\varphi(N_{L^{nr}/K^{nr}}(v_u))
=\varphi(u)=u,
\end{equation*}
combined with the Serre short exact sequence, 
yields the existence of $\sigma_u\in\text{Gal}(L/K)^{ab}$ satisfying
\begin{equation*}
\theta_o(\sigma_u)=\frac{\sigma_u(\pi_L)}{\pi_L}=\frac{v_u}{\varphi(v_u)}.
\end{equation*}

The Hazewinkel mapping
\begin{equation*}
h_{L/K}:U_K/N_{L/K}U_L\rightarrow\text{Gal}(L/K)^{ab}
\end{equation*}
is then defined by
\begin{equation*}
h_{L/K}:u\mapsto\sigma_u,
\end{equation*}
for every $u\in U_K$.

It turns out that, for $L/K$ totally-ramified and finite Galois
extension, the Hazewinkel mapping 
$h_{L/K}:U_K/N_{L/K}U_L\rightarrow\text{Gal}(L/K)^{ab}$ and 
the Iwasawa-Neukirch mapping 
$\iota_{L/K}:\text{Gal}(L/K)\rightarrow K^\times/N_{L/K}(L^\times)$
are the inverses of each other. Thus, by the uniqueness of the local Artin
reciprocity map $\alpha_K:\widehat{K^\times}\rightarrow
G_K^{ab}$ of the local field $K$, it follows that the Hazewinkel map, 
the Iwasawa-Neukirch map,
and the local Artin map are related with each other as 
\begin{equation*}
h_{L/K}=\alpha_{L/K}
\end{equation*}
and
\begin{equation*}
\iota_{L/K}=\alpha_{L/K}^{-1}.
\end{equation*}
\section{Review of ramification theory}
In this section, we shall review the higher-ramification subgroups in
upper-numbering of the absolute Galois group $G_K$ of the local field
$K$, which is necessary in the theory of $APF$-extensions over $K$.
The main reference that we follow for this section is \cite{serre}.

For a finite separable extension $L/K$, and for any
$\sigma\in\text{Hom}_K(L,K^{sep})$, introduce
\begin{equation*}
i_{L/K}(\sigma):=\min\limits_{x\in O_L}
\left\{{\pmb{\nu}}_L(\sigma(x)-x)\right\},
\end{equation*}
put
\begin{equation*}
\gamma_t:=\#\left\{\sigma\in\text{Hom}_K(L,K^{sep}):~i_{L/K}(\sigma)\geq
t+1\right\},
\end{equation*}
for $-1\leq t\in\mathbb R$, and define the function 
$\varphi_{L/K}:\mathbb R_{\geq -1}\rightarrow\mathbb R_{\geq -1}$, 
the \textit{Hasse-Herbrand transition function of the extension $L/K$}, by
\begin{equation*}
\varphi_{L/K}(u):=
\begin{cases}\int_0^u\frac{\gamma_t}{\gamma_0}dt, & 0\leq u\in\mathbb R,\\
u, & -1\leq u\leq 0.\end{cases}
\end{equation*}
It is well-known that, $\varphi_{L/K}:\mathbb R_{\geq -1}\rightarrow\mathbb
R_{\geq -1}$ is a continuous, monotone-increasing, piecewise linear
function, and
induces a homeomorphism 
$\mathbb R_{\geq -1}\xrightarrow{\approx}\mathbb R_{\geq -1}$. 
Now, let $\psi_{L/K}:\mathbb R_{\geq -1}\rightarrow\mathbb R_{\geq -1}$ 
be the mapping inverse to the function $\varphi_{L/K}:
\mathbb R_{\geq -1}\rightarrow\mathbb R_{\geq -1}$.

Assume that $L$ is a finite Galois extension over $K$ with Galois
group $\text{Gal}(L/K)=:G$. The normal subgroup $G_u$ of $G$ defined by
\begin{equation*}
G_u=\{\sigma\in G:~i_{L/K}(\sigma)\geq u+1\}
\end{equation*}
for $-1\leq u\in\mathbb R$ is called the \textit{$u^{th}$ ramification
group of $G$ in the lower numbering}, and has order $\gamma_u$. 
Note that, there is the inclusion $G_{u'}\subseteq G_u$ for every pair
$-1\leq u,u'\in\mathbb R$ satisfying $u\leq u'$. The
family $\{G_u\}_{u\in\mathbb R_{\geq -1}}$ induces a
filtration on $G$, called the \textit{lower ramification filtration of $G$}.
A \textit{break in the lower ramification filtration} 
$\{G_u\}_{u\in\mathbb R_{\geq -1}}$ of $G$ is
defined to be any number $u\in\mathbb R_{\geq -1}$ satisfying $G_u\neq
G_{u+\varepsilon}$ for every $0<\varepsilon\in\mathbb R$.
The function $\psi_{L/K}=\varphi^{-1}_{L/K}:\mathbb R_{\geq
-1}\rightarrow\mathbb R_{\geq -1}$
induces the \textit{upper ramification filtration} 
$\{G^v\}_{v\in\mathbb R_{\geq -1}}$ on $G$ by setting
\begin{equation*}
G^v:=G_{\psi_{L/K}(v)},
\end{equation*}
or equivalently, by setting
\begin{equation*}
G^{\varphi_{L/K}(u)}=G_u
\end{equation*}
for $-1\leq v,u\in\mathbb R$, where $G^v$ is called the 
\textit{$v^{th}$ upper ramification group of $G$}. A \textit{break in
the upper filtration}
$\{G^v\}_{v\in\mathbb R_{\geq -1}}$ of $G$ is defined to be any number
$v\in\mathbb R_{\geq -1}$ satisfying $G^v\neq G^{v+\varepsilon}$ for every
$0<\varepsilon\in\mathbb R$.
\begin{remark}
The basic properties of lower and upper ramification filtrations on $G$
are as follows:

\noindent 
In what follows, $F/K$ denotes a sub-extension of $L/K$ and $H$ denotes
the Galois group $\text{Gal}(L/F)$ corresponding to the extension $L/F$.
\begin{itemize}
\item[(i)] The lower numbering on $G$ passess well to the subgroup $H$ 
of $G$ in the sense that
\begin{equation*}
H_u=H\cap G_u
\end{equation*}
for $-1\leq u\in\mathbb R$;
\item[(ii)] and if furthermore, $H\triangleleft G$, the upper 
numbering on $G$ passess well to the quotient $G/H$ as
\begin{equation*}
(G/H)^v=G^vH/H
\end{equation*}
for $-1\leq v\in\mathbb R$.
\item[(iii)] The Hasse-Herbrand function and its inverse
satisfy the transitive law 
\begin{equation*}
\varphi_{L/K}=\varphi_{F/K}\circ\varphi_{L/F}
\end{equation*}
and
\begin{equation*}
\psi_{L/K}=\psi_{L/F}\circ\psi_{F/K}.
\end{equation*}
\end{itemize}
\noindent
If $L/K$ is an \textit{infinite} Galois extension with Galois group
$\text{Gal}(L/K)=G$, which is a topological group under the respective 
Krull topology, define 
the upper ramification filtration
$\{G^v\}_{v\in\mathbb R_{\geq -1}}$ on $G$ by the projective limit
\begin{equation}
\label{upper-ramification-infinite-ext}
G^v:=\varprojlim\limits_{K\subseteq F\subset L}\text{Gal}(F/K)^v
\end{equation}
defined over the transition morphisms
$t_F^{F'}(v):\text{Gal}(F'/K)^v\rightarrow\text{Gal}(F/K)^v$, which are
essentially the restriction morphisms from $F'$ to $F$, defined
naturally by the diagram
\begin{equation}
\label{upper-ramification-infinite-ext-transition-function}
\SelectTips{cm}{}\xymatrix{
\text{Gal}(F/K)^v & &
\text{Gal}(F'/K)^v\ar[ll]_{t_F^{F'}(v)}\ar[dddl]^{\text{can.}} \\
 &  \\
 & \\
 &
 \text{Gal}(F'/K)^v\text{Gal}(F'/F)/\text{Gal}(F'/F)
\ar[uuul]^{\txt{isomorphism\\ introduced in $(ii)$}} &
}
\end{equation}
induced from (ii), as $K\subseteq F\subseteq F'\subseteq L$ runs over all
finite Galois extensions $F$ and $F'$ over $K$ inside $L$. 
The topological subgroup $G^v$ of $G$ is called the 
\textit{$v^{th}$ ramification group of $G$ in the upper numbering}.
Note that, there is the inclusion $G^{v'}\subseteq G^v$ for every pair
$-1\leq v,v'\in\mathbb R$ satisfying $v\leq v'$ via the commutativity of
the square 
\begin{equation}
\label{upper-ramification-infinite-ext-inclusion}
\SelectTips{cm}{}\xymatrix{
{\text{Gal}(F/K)^v} & & {\text{Gal}(F'/K)^v}\ar[ll]_-{t_F^{F'}(v)} \\
{\text{Gal}(F/K)^{v'}}\ar[u]^{\text{inc.}} & & {\text{Gal}(F'/K)^{v'}}
\ar[ll]_-{t_F^{F'}(v')}\ar[u]_{\text{inc.}}
}
\end{equation}
for every chain $K\subseteq F\subseteq F'\subseteq L$ of finite Galois
extensions $F$ and $F'$ over $K$ inside $L$.
Observe that
\begin{itemize}
\item[(iv)] $G^{-1}=G$ and $G^0$ is the inertia subgroup of $G$;
\item[(v)] $\bigcap\limits_{v\in\mathbb R_{\geq -1}}G^v=\left<1_G\right>$;
\item[(vi)] $G^v$ is a closed subgroup of $G$, with respect to the Krull
topology, for $-1\leq v\in\mathbb R$.
\end{itemize}
\noindent
In this setting, a number $-1\leq v\in\mathbb R$ is said to be a 
\textit{break in the upper ramification filtration} 
$\{G^v\}_{v\in\mathbb R_{\geq -1}}$ of $G$, if
$v$ is a break in the upper filtration of some finite quotient $G/H$ for
some $H\triangleleft G$. Let $\mathcal B_{L/K}$ denotes the set of all numbers
$v\in\mathbb R_{\geq -1}$, which occur as breaks in the upper ramification
filtration of $G$. Then,
\begin{itemize}
\item[(vii)] \textbf{(Hasse-Arf theorem.)} $\mathcal B_{K^{ab}/K}\subseteq 
\mathbb Z\cap\mathbb R_{\geq -1}$;
\item[(viii)] $\mathcal B_{K^{sep}/K}\subseteq\mathbb Q\cap
\mathbb R_{\geq -1}$.
\end{itemize}
\end{remark}
\section{$APF$-extensions over $K$}
In this section, we shall briefly review a very important class of
algebraic extensions, called the $APF$-extensions, over a
local field $K$ introduced by Fontaine and Wintenberger
(cf. \cite{fontaine-wintenberger-1, fontaine-wintenberger-2} and 
\cite{wintenberger-1983}).
As in the previous section, let $\{G_K^v\}_{v\in\mathbb R_{\geq -1}}$ denote
the upper ramification filtration of the absolute Galois group $G_K$ of
$K$, and let $R^v$ denote the fixed 
field $(K^{sep})^{G_K^v}$ of the $v^{th}$ upper ramification subgroup 
$G_K^v$ of $G_K$ in $K^{sep}$ for $-1\leq v\in\mathbb R$.
\begin{definition}
\label{APF}
An extension $L/K$ is called an \textit{$APF$-extension} ($APF$ is
the shortening for ``arithm{\'e}tiquement profinie''), if one of the
following equivalent conditions is satisfied: 
\begin{itemize}
\item[(i)] $G_K^vG_L$ is open in $G_K$ for every $-1\leq v\in\mathbb R$;
\item[(ii)] $(G_K:G_K^v G_L)<\infty$ for every $-1\leq v\in\mathbb R$;
\item[(iii)] $L\cap R^v$ is a finite extension over $K$ for every
$-1\leq v\in\mathbb R$.
\end{itemize}
\end{definition} 
Note that, if $L/K$ is an $APF$-extension, then 
$[\kappa_L:\kappa_K]<\infty$.

Now, let $L/K$ be an $APF$-extension. Set $G_L^0=G_L\cap G_K^0$, and define
\begin{equation}
\label{apf-hasse-herbrand}
\varphi_{L/K}(v)=
\begin{cases} 
\int_0^v(G_K^0:G_L^0G_K^x)dx,& 0\leq v\in\mathbb R;\\
v, &- 1\leq v\leq 0.
\end{cases}
\end{equation}
Then the map $v\mapsto\varphi_{L/K}(v)$ for $v\in\mathbb R_{\geq -1}$, which
is well-defined for the $APF$-extension $L/K$, defines a continuous, 
strictly-increasing and piecewise-linear
bijection $\varphi_{L/K}:\mathbb R_{\geq -1}\rightarrow\mathbb R_{\geq -1}$.
We denote the inverse of this mapping by
$\psi_{L/K}:=\varphi_{L/K}^{-1}:\mathbb R_{\geq -1}\rightarrow\mathbb
R_{\geq -1}$.

Thus, if $L/K$ is a (not necessarly finite) Galois $APF$-extension, 
then we can define the higher ramification subgroups in lower numbering
$\text{Gal}(L/K)_u$ of $\text{Gal}(L/K)$, for $-1\leq u\in\mathbb R$,
by setting
\begin{equation*}
\text{Gal}(L/K)_u:=\text{Gal}(L/K)^{\varphi_{L/K}(u)} .
\end{equation*}  
\begin{remark}
Note that,
\begin{itemize}
\item[(i)] In case $L/K$ is a finite separable extension, which is clearly an
$APF$-extension by Definition \ref{APF}, the function
$\varphi_{L/K}:\mathbb
R_{\geq -1}\rightarrow\mathbb R_{\geq -1}$ coincides with the Hasse-Herbrand
transition function of $L/K$ introduced in the previous section;
\item[(ii)] if $L/K$ is a finite separable extension and $L'/L$ is an
$APF$-extension, then $L'/K$ is an $APF$-extension, and the transitivity
rules for the functions
$\varphi_{L'/K},\psi_{L'/K}:\mathbb R_{\geq -1}\rightarrow\mathbb R_{\geq -1}$
hold by
\begin{equation*}
\varphi_{L'/K}=\varphi_{L/K}\circ\varphi_{L'/L}
\end{equation*} 
and by
\begin{equation*}
\psi_{L'/K}=\psi_{L'/L}\circ\psi_{L/K}.
\end{equation*}
\end{itemize}
\end{remark}
The following result will be extremely useful.
\begin{lemma}
\label{apftower}
Suppose that $K\subseteq F\subseteq L\subseteq K^{sep}$ is a tower of
field extensions in $K^{sep}$ over $K$. Then,
\begin{itemize}
\item[(i)] If $[F:K]<\infty$, then $L/K$ is an $APF$-extension if and 
only if $L/F$ is an $APF$-extension.
\item[(ii)] If $[L:F]<\infty$, then $L/K$ is an $APF$-extension if and
 only if $F/K$ is an $APF$-extension.
\item[(iii)] If $L/K$ is an $APF$-extension, then 
$F/K$ is an $APF$-extension.
\end{itemize}
\end{lemma}
\begin{proof}
For a proof, look at Proposition 1.2.3. of \cite{wintenberger-1983}.
\end{proof}
\section{Fontaine-Wintenberger fields of norms}
Let $L/K$ be an \textit{infinite $APF$-extension}. Let $L_i$ for $0\leq
i\in\mathbb Z$ be an increasing directed-family of sub-extensions in $L/K$
such that: 
\begin{itemize}
\item[(i)] $[L_i:K]<\infty$ for every $0\leq i\in\mathbb Z$;
\item[(ii)] $\bigcup\limits_{0\leq i\in\mathbb Z}L_i=L$.
\end{itemize}
Let 
\begin{equation*}
\mathbb X(L/K)^\times=\varprojlim\limits_{i}L_i^\times
\end{equation*}
be the
projective limit of the multiplicative groups $L_i^\times$ with respect to
the norm homomorphisms 
\begin{equation*}
N_{L_{i'}/L_i}:L_{i'}^\times\rightarrow L_i^\times, 
\end{equation*}
for every $0\leq i,i'\in\mathbb Z$ with $i\leq i'$.
\begin{remark} 
The group $\mathbb X(L/K)^\times$ does not depend on the choice of the
increasing directed-family of sub-extensions $\{L_i\}_{0\leq i\in\mathbb Z}$
in $L/K$ satisfying the conditions $(i)$ and $(ii)$.
Thus, 
\begin{equation*}
\mathbb X(L/K)^\times=\varprojlim\limits_{M\in S_{L/K}}M^\times,
\end{equation*}
where $S_{L/K}$ is the partially-ordered family of all finite
sub-extensions in $L/K$, and the projective limit is with respect to
the norm 
\begin{equation*}
N_{M_2/M_1}:M_2^\times\rightarrow M_1^\times,
\end{equation*} 
for every $M_1,M_2\in S_{L/K}$ with $M_1\subseteq M_2$.
\end{remark}
Put 
\begin{equation*}
\mathbb X(L/K)=\mathbb X(L/K)^\times\cup\{0\},
\end{equation*}
where $0$ is a fixed symbol, and define the addition
\begin{equation*}
+:\mathbb X(L/K)\times\mathbb X(L/K)\rightarrow\mathbb X(L/K)
\end{equation*}
by the rule
\begin{equation*}
(\alpha_M)+(\beta_M)=(\gamma_M),
\end{equation*}
where $\gamma_M\in M$ is defined by the limit
\begin{equation}
\label{addition}
\gamma_M=\lim_{\substack{ M\subset M'\in S_{L/K}\\
[M':M]\rightarrow\infty}}N_{M'/M}(\alpha_{M'}+\beta_{M'}),
\end{equation} 
which exists in the local field $M$, for every $M\in S_{L/K}$.
\begin{remark}
Note that, for $(\alpha_M), (\beta_M)\in\mathbb X(L/K)$, the law of
composition 
\begin{equation*}
\left((\alpha_M),(\beta_M)\right)\mapsto (\alpha_M)+(\beta_M)=(\gamma_M)
\end{equation*}
given by eq. (\ref{addition}) is well-defined, since $L/K$ is
assumed to be an $APF$-extension (cf. Theorem 2.1.3. of 
\cite{wintenberger-1983}).
\end{remark} 
It then follows that, 
\begin{theorem}[Fontaine-Wintenberger]
Let $L/K$ be an $APF$-extension. Then $\mathbb X(L/K)$ is a field under the
addition
\begin{equation*}
+:\mathbb X(L/K)\times\mathbb X(L/K)\rightarrow\mathbb X(L/K),
\end{equation*}
defined by eq. (\ref{addition}), and under the multiplication
\begin{equation*}
\times :\mathbb X(L/K)\times\mathbb X(L/K)\rightarrow\mathbb X(L/K)
\end{equation*} 
defined naturally from the componentwise multiplication defined on 
$\mathbb X(L/K)^\times$. 
This field $\mathbb X(L/K)$ is called the 
{\it field of norms corresponding to the $APF$-extension $L/K$}.  
\end{theorem}
Now, in particular, choose the following specific increasing
directed-family of sub-extensions $\{L_i\}_{0\leq i\in\mathbb Z}$ in
$L/K$ :
\begin{itemize}
\item[(i)] $L_0$ is the maximal unramified extension of $K$ inside $L$;
\item[(ii)] $L_1$ is the maximal tamely ramified extension of $K$ 
inside $L$;
\item[(iii)] choose $L_i$, for $i\geq 2$, inductively as a finite
extension of $L_1$ inside $L$ with $L_i\subseteq L_{i+1}$ and 
$\bigcup\limits_{0\leq i\in\mathbb Z}L_i=L$.
\end{itemize}
Note that, $L_0/K$ is a finite sub-extension of $L/K$ and
by the definition of tamely ramified extensions, $L_0\subseteq L_1$,
with $[L_1:K]<\infty$.
Thus, for any element $(\alpha_{L_i})_{0\leq i\in\mathbb Z}$ of
$\mathbb X(L/K)$,
\begin{equation}
\label{dv}
\nu_{L_i}(\alpha_{L_i})=\nu_{L_0}(\alpha_{L_0}),
\end{equation}
for every $0\leq i\in\mathbb Z$. Thus, the mapping
\begin{equation*} 
\pmb{\nu}_{\mathbb X(L/K)}:\mathbb X(L/K)\rightarrow\mathbb Z\cup\{\infty\}
\end{equation*}
given by
\begin{equation}
\label{valuationdef}
\pmb{\nu}_{\mathbb X(L/K)}\left((\alpha_{L_i})_{0\leq
i\in\mathbb Z}\right)=\nu_{L_0}(\alpha_{L_0}),
\end{equation}
for $(\alpha_{L_i})_{0\leq i\in\mathbb Z}\in\mathbb X(L/K)$, is
well-defined, 
and moreover a discrete valuation on $\mathbb X(L/K)$, in view 
of eq. (\ref{dv}).
\begin{theorem}[Fontaine-Wintenberger]
\label{completefield}
Let $L/K$ be an $APF$-extension, and let $\mathbb X(L/K)$ be the field of
norms attached to $L/K$. Then,
\begin{itemize}
\item[(i)] the field $\mathbb X(L/K)$ is complete with respect to the 
discrete valuation 
$\pmb{\nu}_{\mathbb X(L/K)}:\mathbb X(L/K)\rightarrow\mathbb
Z\cup\{\infty\}$ defined by eq. (\ref{valuationdef});
\item[(ii)] the residue class field $\kappa_{\mathbb X(L/K)}$ of 
$\mathbb X(L/K)$ satisfies 
$\kappa_{\mathbb X(L/K)}\xrightarrow{\sim}\kappa_L$;
\item[(iii)] the characteristic of the field $\mathbb X(L/K)$ is equal to 
$\text{char}(\kappa_K)$.
\end{itemize}
\end{theorem}
\begin{proof}
For a proof, look at Theorem 2.1.3 of \cite{wintenberger-1983}.
\end{proof}
\begin{remark}
The ring of integers $O_{{\mathbb X}(L/K)}$ of the local field
(complete discrete valuation field) $\mathbb X(L/K)$ is defined
as usual by
\begin{equation*}
O_{{\mathbb X}(L/K)}=\left\{(\alpha_{L_i})_{0\leq i\in\mathbb Z}
\in\mathbb X(L/K): 
\pmb{\nu}_{\mathbb X(L/K)}
\left((\alpha_{L_i})_{0\leq i\in\mathbb Z}\right)\geq 0\right\}.
\end{equation*}
Thus, by eq.s (\ref{valuationdef}) and (\ref{dv}), for 
$\alpha=(\alpha_{L_i})_{0\leq i\in\mathbb Z}\in {\mathbb X}(L/K)$,
the following two conditions are equivalent.
\begin{itemize}
\item[(i)]
$(\alpha_{L_i})_{0\leq i\in\mathbb Z}\in O_{\mathbb X(L/K)}$; 
\item[(ii)]
$\alpha_{L_i}\in O_{L_i}$ for every
$0\leq i\in\mathbb Z$.
\end{itemize}
The maximal ideal $\mathfrak p_{\mathbb X(L/K)}$ of $O_{\mathbb
 X(L/K)}$ is defined by
\begin{equation*}
\mathfrak p_{\mathbb X(L/K)}=
\left\{(\alpha_{L_i})_{0\leq i\in\mathbb Z}\in\mathbb X(L/K): 
\pmb{\nu}_{\mathbb X(L/K)}
\left((\alpha_{L_i})_{0\leq i\in\mathbb Z}\right)\gneq 0\right\}.
\end{equation*}
By eq.s (\ref{valuationdef}) and (\ref{dv}), for
$\alpha=(\alpha_{L_i})_{0\leq i\in\mathbb Z}\in {\mathbb X}(L/K)$,
the following two conditions are equivalent.
\begin{itemize}
\item[(iii)]
$(\alpha_{L_i})_{0\leq i\in\mathbb Z}\in\mathfrak p_{\mathbb X(L/K)}$;
\item[(iv)]
$\alpha_{L_i}\in\mathfrak p_{L_i}$ for every
$0\leq i\in\mathbb Z$.
\end{itemize}
The unit group $U_{\mathbb X(L/K)}$ of $O_{\mathbb X(L/K)}$ is defined
by
\begin{equation*}
U_{\mathbb X(L/K)}=
\left\{(\alpha_{L_i})_{0\leq i\in\mathbb Z}\in\mathbb X(L/K): 
\pmb{\nu}_{\mathbb X(L/K)}
\left((\alpha_{L_i})_{0\leq i\in\mathbb Z}\right)= 0\right\}.
\end{equation*}
Again by eq.s (\ref{valuationdef}) and (\ref{dv}), for
$\alpha=(\alpha_{L_i})_{0\leq i\in\mathbb Z}\in {\mathbb X}(L/K)$,
the following two conditions are equivalent.
\begin{itemize}
\item[(v)]
$(\alpha_{L_i})_{0\leq i\in\mathbb Z}\in U_{\mathbb X(L/K)}$;
\item[(vi)]
$\alpha_{L_i}\in U_{L_i}$ for every $0\leq i\in\mathbb Z$.
\end{itemize}
\end{remark}
Let $L/K$ be an infinite $APF$-extension. Consider the following tower
\begin{equation*}
K\subseteq F\subseteq L\subseteq E\subseteq K^{sep}
\end{equation*}
of extensions over $K$, where $[F:K]<\infty$ and $[E:L]<\infty$. 
It then follows, by Lemma \ref{apftower} parts (i) and (ii), that 
$L/F$ is an infinite $APF$-extension satisfying
\begin{equation*}
\mathbb X(L/K)=\mathbb X(L/F),
\end{equation*}
by the definition of field of norms, and $E/K$ is an infinite 
$APF$-extension satisfying
\begin{equation*}
\mathbb X(L/K)\hookrightarrow\mathbb X(E/K)
\end{equation*}
under the injective topological homomorphism
\begin{equation*}
\varepsilon_{L,E}^{(M)} : \mathbb X(L/K)\rightarrow\mathbb X(E/K),
\end{equation*}
which depends on a finite extension $M$ over $K$ satisfying $LM=E$.
\begin{equation*}
\xymatrix{
{} & {LM=E}\\
L\ar@{-}[ur] & {}\\
{} & M\ar@{-}[uu]\\
K\ar@{-}[ur]_-{[M:K]<\infty}\ar@{-}[uu]^-{\txt{infinite\\APF-ext.}} & {}
}
\end{equation*}
The topological embedding $\varepsilon_{L,E}^{(M)}:\mathbb
X(L/K)\hookrightarrow\mathbb X(E/K)$ is defined as follows. 
Let $\{L_i\}_{0\leq i\in\mathbb Z}$ be an increasing directed-family
of sub-extensions in $L/K$, such that $[L_i:K]<\infty$, for every 
$0\leq i\in\mathbb Z$, and $\bigcup_{0\leq i\in\mathbb Z}L_i=L$. Then,
clearly, $\{L_iM\}_{0\leq i\in\mathbb Z}$ is an increasing
directed-family of sub-extensions in $E/K$, such that
$[L_iM:K]<\infty$, for every $0\leq i\in\mathbb Z$, and 
$\bigcup_{0\leq i\in\mathbb Z}L_iM=E$. Given these two
directed-families, there exists a large
enough positive integer $m=m(M)$, which depends on the choice of $M$, 
such that, for $m\leq i\leq j$,
\begin{equation*}
N_{L_jM/L_iM}(x)=N_{L_j/L_i}(x),
\end{equation*} 
for each $x\in L_j$.
Now, the topological embedding 
$\varepsilon_{L,E}^{(M)} : \mathbb X(L/K)\hookrightarrow\mathbb X(E/K)$ is
defined,
for every $(\alpha_{L_i})_{0\leq i\in\mathbb Z}
\in\mathbb X(L/K)-\{0\}$, by
\begin{equation*}
\varepsilon_{L,E}^{(M)} : (\alpha_{L_i})_{0\leq i\in\mathbb Z}\mapsto 
(\alpha'_{L_iM})_{0\leq i\in\mathbb Z},
\end{equation*}
where $\alpha'_{L_iM}\in L_iM$ is defined, for every 
$0\leq i\in\mathbb Z$, by
\begin{equation*}
\alpha'_{L_iM}=
\begin{cases}
\alpha_{L_i},& i\geq m \\
N_{L_mM/L_iM}(\alpha_{L_m}), & i<m .
\end{cases}
\end{equation*}
Thus, under the topological embedding 
$\varepsilon_{L,E}^{(M)}:\mathbb X(L/K)\hookrightarrow\mathbb X(E/K)$,
view $\mathbb X(E/K)/\mathbb X(L/K)$ as an extension of complete
discrete valuation fields. 
At this point, the following remark is in order.
\begin{remark}
Let $L/K$ be an infinite $APF$-extension and $E/L$ a finite extension.
Suppose that $M$ and $M'$ are two finite extensions over $K$, satisfying
$LM=LM'=E$. Then the embeddings 
$\varepsilon_{L,E}^{(M)}, \varepsilon_{L,E}^{(M')} :
\mathbb X(L/K)\hookrightarrow\mathbb X(E/K)$ are the same. Therefore,
as a notation, we shall set $\varepsilon_{L,E}^{(M)}=\varepsilon_{L,E}$.
\end{remark}
Now, given an infinite $APF$-extension $L/K$, and this time let
$E$ be a (not necessarily finite) separable extension of $L$. Let
$S_{E/L}^{sep}$ denote the partially-ordered family of all finite separable 
sub-extensions in $E/L$. Then,
\begin{proposition}
\begin{equation*}
\left\{\mathbb X(E'/K);
\varepsilon_{E',E''}:\mathbb X(E'/K)\hookrightarrow\mathbb
X(E''/K)\right\}_{\substack{E',E''\in S_{E/L}^{sep}\\ E'\subseteq E''}}
\end{equation*}
is an inductive system under the topological embeddings
\begin{equation*}
\varepsilon_{E',E''}:\mathbb X(E'/K)\hookrightarrow
\mathbb X(E''/K)
\end{equation*} 
for $E',E''\in S_{E/L}^{sep}$ with $E'\subseteq E''$. 
\end{proposition}
Let $\mathbb X(E,L/K)$ denote the topological field defined by
the inductive limit 
\begin{equation*}
\mathbb X(E,L/K)=\varinjlim\limits_{E'\in S_{E/L}^{sep}}\mathbb X(E'/K)
\end{equation*}
defined over the
transition morphisms $\varepsilon_{E',E''}:\mathbb X(E'/K)\hookrightarrow
\mathbb X(E''/K)$ for $E',E''\in S_{E/L}^{sep}$ with $E'\subseteq E''$.

The following theorem is central in the theory of fields of norms.
\begin{theorem}[Fontaine-Wintenberger]
Let $L/K$ be an $APF$-extension and $E/L$ a Galois extension. Then
$\mathbb X(E,L/K)/\mathbb X(L/K)$ is a Galois extension, and
\begin{equation*}
\text{Gal}\left(\mathbb X(E,L/K)/\mathbb X(L/K)\right)
\simeq\text{Gal}(E/L)
\end{equation*}
canonically.
\end{theorem}
An immediate and important consequence of this theorem is the following.
\begin{corollary}
Let $L/K$ be an $APF$-extension. Then
\begin{equation*}
\text{Gal}(\mathbb X(L^{sep},L/K)/\mathbb X(L/K))
\simeq\text{Gal}(L^{sep}/L)
\end{equation*}
canonically.
\end{corollary}
\section{Fesenko reciprocity law}
In this section, we shall review the Fesenko reciprocity law
for the local field $K$ following \cite{fesenko-2000, fesenko-2001, 
fesenko-2005}.

Following \cite{koch-deshalit}, we recall the following definition.
\begin{definition}
Let $\varphi=\varphi_K\in\text{Gal}(K^{nr}/K)$ denote the Frobenius 
automorphism of $K$. An automorphism $\xi\in\text{Gal}(K^{sep}/K)$ is called a 
\textit{Lubin-Tate splitting over $K$},
if $\xi\mid_{K^{nr}}=\varphi$.
\end{definition}
All through the remainder of the text, we shall fix a Lubin-Tate
splitting over the local field $K$ and denote it simply by $\varphi$, or by
$\varphi_K$ if there is fear of confusion.
Let $K_\varphi$ denote the fixed field $(K^{sep})^{\varphi}$ 
of $\varphi\in G_K$ in $K^{sep}$. 

Let $L/K$ be a totally-ramified $APF$-Galois extension satisfying 
\begin{equation}
\label{phitower}
K\subseteq L\subseteq K_\varphi. 
\end{equation}
The field of norms $\mathbb X(L/K)$ is a local field by
virtue of Theorem \ref{completefield}.
Let $\widetilde{\mathbb X}(L/K)$ denote the completion 
$\widetilde{\mathbb X(L/K)}$ of 
$\mathbb X(L/K)^{nr}$ with respect to the valuation
$\pmb{\nu}_{\mathbb X(L/K)^{nr}}$, which 
is the unique extension of the valuation 
$\pmb{\nu}_{\mathbb X(L/K)}$ to $\mathbb X(L/K)^{nr}$.
As usual, we let $U_{\widetilde{\mathbb X}(L/K)}$ to denote the unit
group of the ring of integers $O_{\widetilde{\mathbb X}(L/K)}$ 
of the complete field $\widetilde{\mathbb X}(L/K)$. 
In this case, there exist isomorphisms 
\begin{equation*}
\widetilde{\mathbb X}(L/K)\simeq\mathbb F_p^{sep}((T))
\end{equation*}
and 
\begin{equation*}
U_{\widetilde{\mathbb X}(L/K)}\simeq\mathbb F_p^{sep}[[T]]^\times
\end{equation*}
defined by the mechanism of Coleman power series (for details, q.v. section 1.4
in \cite{koch-deshalit}).
Thus, the algebraic structures $\widetilde{\mathbb X}(L/K)$ and 
$U_{\widetilde{\mathbb X}(L/K)}$ initially seems to depend on the ground 
field $K$ only. However, as we shall state in Corollary \ref{grouplaw1},
the law of composition on the ``class formation'', which is a certain 
sub-quotient of $U_{\widetilde{\mathbb X}(L/K)}$, does indeed depend on 
the $Gal(L/K)$-module structure of $U_{\widetilde{\mathbb X}(L/K)}$. 
\begin{remark}
The problem of removing this dependence on the Galois-module structure of
$U_{\widetilde{\mathbb X}(L/K)}$ is closely connected with Sen's 
infinite-dimensional Hodge-Tate theory (\cite{ikeda-serbest}), or more 
generally with the $p$-adic Langlands program.
\end{remark}
As in section 1, let $\widetilde{K}$ denote the completion of $K^{nr}$
with respect to the valuation $\nu_{K^{nr}}$ on $K^{nr}$, and let
$\widetilde{L}=L\widetilde{K}$. Then $\widetilde{L}/\widetilde{K}$ is
an $APF$-extension, as $L/K$ is an $APF$-extension, and the
corresponding field of norms satisfy
\begin{equation}
\label{identify}
\mathbb X(\widetilde{L}/\widetilde{K})=\widetilde{\mathbb X}(L/K).
\end{equation}
Now, let
\begin{equation}
\label{projection-to-K}
\text{Pr}_{\widetilde K}:U_{\widetilde{\mathbb X}(L/K)}\rightarrow
U_{\widetilde K}
\end{equation}
denote the projection map on the $\widetilde{K}$-coordinate of 
$U_{\widetilde{\mathbb X}(L/K)}$ under the identification described in
eq. (\ref{identify}).
All through the text, $U_{\widetilde{\mathbb X}(L/K)}^1$ stands for
the kernel $\text{ker}(\text{Pr}_{\widetilde K})$ of 
the projection map $\text{Pr}_{\widetilde K}:U_{\widetilde{\mathbb X}(L/K)}
\rightarrow U_{\widetilde K}$. 

\begin{definition}
The subgroup 
\begin{equation*}
\text{Pr}_{\widetilde K}^{-1}(U_K)=\{U\in 
U_{\widetilde{\mathbb X}{(L/K)}}:\text{Pr}_{\widetilde K}(U)\in U_K\} 
\end{equation*}
of $U_{\widetilde{\mathbb X}{(L/K)}}$ is called the 
\textit{Fesenko diamond subgroup of $U_{\widetilde{\mathbb X}{(L/K)}}$}, 
and denoted by $U_{\widetilde{\mathbb X}{(L/K)}}^\diamond$. 
\end{definition}
Now, following \cite{fesenko-2000, fesenko-2001, fesenko-2005}, choose an 
ascending chain of field extensions
\begin{equation*}
K=E_o\subset E_1\subset\dots\subset E_i\subset\dots\subset L
\end{equation*} 
in such a way that 
\begin{itemize}
\item[(i)]
$L=\bigcup_{0\leq i\in\mathbb Z}E_i$ ;
\item[(ii)]
$E_i/K$ is a Galois extension for each $0\leq i\in\mathbb Z$; 
\item[(iii)]
$E_{i+1}/E_i$ is cyclic of prime degree $[E_{i+1}:E_i]=p
=\text{char}(\kappa_K)$ for each $1\leq i\in\mathbb Z$;
\item[(iv)]
$E_1/E_o$ is cyclic of degree relatively prime to $p$. 
\end{itemize}
Such a sequence $(E_i)_{0\leq i\in\mathbb Z}$
exists, as $L/K$ is a solvable Galois extension, and will be called as
a \textit{basic ascending chain of sub-extensions in $L/K$}.
Then, we can
construct $\mathbb X(L/K)$ by the basic sequence 
$(E_i)_{0\leq i\in\mathbb Z}$ and
$\widetilde{\mathbb X}{(L/K)}$ by $(\widetilde{E_i})_{0\leq
i\in\mathbb Z}$. Note that, the Galois group $\text{Gal}(L/K)$
corresponding to the extension $L/K$ act continuously on 
$\mathbb X(L/K)$ and on
$\widetilde{\mathbb X}{(L/K)}$ naturally by defining the
Galois-action of $\sigma\in\text{Gal}(L/K)$ on the chain
\begin{equation}
\label{galoismodule-definition}
K=E_o\subset E_1\subset\dots\subset E_i\subset
\dots\subset L,
\end{equation}
by the action of $\sigma$ on each $E_i$ for $0\leq i\in\mathbb Z$ as
\begin{equation}
\label{galoismodule-action}
K=E_o^\sigma\subset E_1^\sigma=E_1\subset\dots\subset E_i^\sigma=E_i\subset
\dots\subset L,
\end{equation}
and respectively on the chain
\begin{equation}
\label{galoismodule-definition-2}
\widetilde{K}=\widetilde{K}E_o\subset\widetilde{E}_1=\widetilde{K}E_1\subset
\dots\subset\widetilde{E}_i=\widetilde{K}E_i\subset\dots\subset
\widetilde{L}=\widetilde{K}L,
\end{equation}
by the action of $\sigma$ on the ``$E_i$-part'' of each $\widetilde{E}_i$
(note that, $E_i\cap K^{nr}=K$) 
\begin{equation*}
\begin{matrix}
\xymatrix{
{} & {\widetilde{K}E_i=\widetilde{E}_i}\\
{\widetilde{K}}\ar@{-}[ur] & {}\\
{} & {E_i}\ar@{-}[uu]\\
K\ar@{-}[ur]_-{\txt{$E_i/K$ Galois ext. \\ $[E_i:K]<\infty$}}
\ar@{-}[uu]^-{\txt{completion of\\ max.-ur.-ext. of K}} & {}
} &
\xymatrix{ {} \\ \rightsquigarrow \\ \rightsquigarrow \\
  \rightsquigarrow \\ {} } &
\xymatrix{
{} & {\widetilde{K}E_i^{\sigma}=\widetilde{E_i^{\sigma}}=\widetilde{E}_i}\\
{\widetilde{K}}\ar@{-}[ur] & {}\\
{} & {E_i^\sigma = E_i}\ar@{-}[uu]\\
K\ar@{-}[ur]_
-{\txt{$E_i^\sigma/K$ Galois ext. \\ $[E_i^{\sigma}:K]<\infty$}}
\ar@{-}[uu]^-{\txt{completion of \\ max.-ur.-ext. of K}} & {}
}
\end{matrix}
\end{equation*}
for $0\leq i\in\mathbb Z$ as
\begin{equation}
\label{galoismodule-action-2}
\widetilde{K}=\widetilde{K}E_o^\sigma\subset
\widetilde{K}E_1^\sigma=\widetilde{K}E_1\subset\dots\subset 
\widetilde{K}E_i^\sigma=\widetilde{K}E_i\subset
\dots\subset\widetilde{K}L.
\end{equation}
Therefore, there exist natural
continuous actions of $\text{Gal}(L/K)$ on
$U_{\mathbb X(L/K)}$, $U_{\widetilde{\mathbb X}{(L/K)}}$, and on
$U_{\widetilde{\mathbb X}{(L/K)}}^\diamond$ compatible with the
respective topological group structures, so that \textit{we shall always 
view them as topological $\text{Gal}(L/K)$-modules in this text}.
Now, recall the following theorem regarding norm compatible sequences
of prime elements (cf. \cite{koch-deshalit}).
\begin{theorem}[Koch-de Shalit]
Assume that $K\subseteq L\subset K_\varphi$. Then for any chain
\begin{equation*}
K=E_o\subset E_1\subset\dots\subset E_i\subset\dots\subset L,
\end{equation*}
of finite sub-extensions of $L/K$,
there exists a unique norm-compatible sequence 
\begin{equation*}
\pi_{E_o}, \pi_{E_1},\cdots,\pi_{E_i},\cdots , 
\end{equation*}
where each $\pi_{E_i}$ is a prime element of $E_i$,
for $0\leq i\in\mathbb Z$.
\end{theorem}
In view of the theorem of Koch and de Shalit, define the 
\textit{natural} prime element
$\Pi_{\varphi;L/K}$ of the local field $\mathbb X(L/K)$, 
which depends on the fixed Lubin-Tate splitting $\varphi$ 
(cf. \cite{koch-deshalit}) as well as
the sub-extension $L/K$ of $K_\varphi/K$, by
\begin{equation*}
\Pi_{\varphi;L/K}=(\pi_{E_i})_{0\leq i\in\mathbb Z}.
\end{equation*}
Note that, by the theorem of Koch and de Shalit, the prime element
$\Pi_{\varphi;L/K}$ of $\mathbb X(L/K)$ does \textit{not} depend on the 
choice of a chain $(E_i)_{0\leq i\in\mathbb Z}$ of finite 
sub-extensions of $L/K$. 
\begin{theorem}[Fesenko]
\label{fundamental-equation}
For each $\sigma\in\text{Gal}(L/K)$, there exists 
$U_\sigma\in U_{\widetilde{\mathbb X}{(L/K)}}^\diamond$ which solves 
the equation
\begin{equation}
U^{1-\varphi}=\Pi_{\varphi;L/K}^{\sigma-1}
\end{equation}
for $U$. Moreover, the solution set of this equation consists of
elements from the coset $U_\sigma .U_{\mathbb X(L/K)}$ of $U_\sigma$
modulo $U_{\mathbb X(L/K)}$. 
\end{theorem}
In fact, for the most general form of this theorem and its proof,
look at \cite{ikeda-ikeda}. 

Now, define the arrow
\begin{equation}
\phi_{L/K}^{(\varphi)}:\text{Gal}(L/K)\rightarrow 
U_{\widetilde{\mathbb X}{(L/K)}}^\diamond/U_{\mathbb X(L/K)}
\end{equation}
by
\begin{equation}
\phi_{L/K}^{(\varphi)}:\sigma\mapsto
\overline{U}_\sigma=U_\sigma.U_{\mathbb X(L/K)},
\end{equation}
for every $\sigma\in\text{Gal}(L/K)$.
\begin{theorem}[Fesenko]
\label{cocycle1}
The arrow
\begin{equation*}
\phi_{L/K}^{(\varphi)}:\text{Gal}(L/K)\rightarrow 
U_{\widetilde{\mathbb X}{(L/K)}}^\diamond/U_{\mathbb X(L/K)}
\end{equation*}
defined for the extension $L/K$
is injective, and for every $\sigma,\tau\in\text{Gal}(L/K)$, 
\begin{equation}
\phi_{L/K}^{(\varphi)}(\sigma\tau)=\phi_{L/K}^{(\varphi)}(\sigma)
\phi_{L/K}^{(\varphi)}(\tau)^\sigma
\end{equation}
co-cycle condition is satisfied.
\end{theorem}
A natural consequence of this theorem is the following result.
Let $\text{im}(\phi_{L/K}^{(\varphi)})\subseteq 
U_{\widetilde{\mathbb X}{(L/K)}}^\diamond/U_{\mathbb X(L/K)}$ denote the
image set of the mapping $\phi_{L/K}^{(\varphi)}$.
\begin{corollary}
\label{grouplaw1}
Define a law of composition $\ast$ on
$\text{im}(\phi_{L/K}^{(\varphi)})$ by
\begin{equation}
\label{star1}
\overline{U}\ast\overline{V}=\overline{U}.
\overline{V}^{(\phi_{L/K}^{(\varphi)})^{-1}(\overline{U})}
\end{equation}
for every
$\overline{U},\overline{V}\in\text{im}(\phi_{L/K}^{(\varphi)})$. 
Then
$\text{im}(\phi_{L/K}^{(\varphi)})$ is a topological group under $\ast$, and
the map $\phi_{L/K}^{(\varphi)}$ induces an isomorphism of topological 
groups
\begin{equation}
\phi_{L/K}^{(\varphi)}:\text{Gal}(L/K)\xrightarrow{\sim}
\text{im}(\phi_{L/K}^{(\varphi)}),
\end{equation}
where the topological group structure on 
$\text{im}(\phi_{L/K}^{(\varphi)})$ is
defined with respect to the binary operation $\ast$ defined by
eq. (\ref{star1}).
\end{corollary}
Now, for each $0\leq i\in\mathbb R$, consider the $i^{th}$ higher unit
group $U_{\widetilde{\mathbb X}{(L/K)}}^i$ of the field 
$\widetilde{\mathbb X}{(L/K)}$, and define the group
\begin{equation}
\left(U_{\widetilde{\mathbb X}{(L/K)}}^\diamond\right)^i=
U_{\widetilde{\mathbb X}{(L/K)}}^\diamond\cap 
U_{\widetilde{\mathbb X}{(L/K)}}^i .
\end{equation}
\begin{theorem}[Fesenko ramification theorem]
\label{ramification1}
For $0\leq n\in\mathbb Z$,
let $\text{Gal}(L/K)_n$ denote the $n^{th}$ higher ramification 
subgroup of the Galois group $\text{Gal}(L/K)$ corresponding to the
$APF$-Galois sub-extension $L/K$ of $K_\varphi/K$ in the lower 
numbering. Then, there exists the inclusion
\begin{multline*}
\phi_{L/K}^{(\varphi)}
\left(\text{Gal}(L/K)_n-\text{Gal}(L/K)_{n+1}\right)\subseteq \\
\left(U_{\widetilde{\mathbb X}{(L/K)}}^\diamond\right)^nU_{\mathbb
  X(L/K)}/U_{\mathbb X(L/K)} -
\left(U_{\widetilde{\mathbb X}{(L/K)}}^\diamond\right)^{n+1}
U_{\mathbb X(L/K)}/U_{\mathbb X(L/K)}.
\end{multline*}
\end{theorem}
Now, let $M/K$ be a Galois sub-extension of $L/K$. Thus, there exists
the chain of field extensions
\begin{equation*}
K\subseteq M\subseteq L\subseteq K_\varphi ,
\end{equation*}
where $M$ is a totally-ramified $APF$-Galois extension over $K$ by
Lemma \ref{apftower}. Let
\begin{equation*}
\phi_{M/K}^{(\varphi)}:\text{Gal}(M/K)\rightarrow
U_{\widetilde{\mathbb X}{(M/K)}}^\diamond/U_{\mathbb X(M/K)}
\end{equation*}
be the corresponding map defined for the extension $M/K$.

Now, let 
\begin{equation*}
K=E_o\subset E_1\subset\cdots\subset E_i\subset\cdots\subset L
\end{equation*}
be an ascending chain
satisfying $L=\bigcup_{0\leq i\in\mathbb Z}E_i$ and
$[E_{i+1}:E_i]<\infty$ for every $0\leq i\in\mathbb Z$. Then
\begin{equation*}
K=E_o\cap M\subseteq E_1\cap M\subseteq\cdots\subseteq E_i\cap
M\subseteq\cdots\subset M
\end{equation*}
is an ascending chain of field extensions satisfying
$M=\bigcup_{0\leq i\in\mathbb Z}(E_i\cap M)$ and 
$[E_{i+1}\cap M:E_i\cap M]<\infty$ for every $0\leq i\in\mathbb Z$.
Thus, we can construct ${\mathbb X}(M/K)$ by the sequence $(E_i\cap
M)_{0\leq i\in\mathbb Z}$ and $\widetilde{\mathbb X}(M/K)$ by 
the sequence $(\widetilde{E_i\cap M})_{0\leq i\in\mathbb Z}$. 
Furthermore, the commutative square
\begin{equation*}
\SelectTips{cm}{}\xymatrix{
{\widetilde{E}_i^\times}\ar[d]_{\widetilde{N}_{E_i/E_i\cap M}}  & & 
{\widetilde{E}_{i'}^\times}\ar[ll]_{\widetilde{N}_{E_{i'}/E_i}}
\ar[d]^{\widetilde{N}_{E_{i'}/E_{i'}\cap M}} \\
{\widetilde{E_i\cap M}^\times} & &
{\widetilde{E_{i'}\cap M}^\times}\ar[ll]_{\widetilde{N}_{E_{i'}\cap
    M/E_i\cap M}}
}
\end{equation*}
for every pair $0\leq i,i'\in\mathbb Z$ satisfying $i\leq i'$, induces
a group homomorphism
\begin{equation}
\label{coleman-norm}
\widetilde{\mathcal N}_{L/M}=\varprojlim\limits_{0\leq i\in\mathbb Z}
\widetilde{N}_{E_i/E_i\cap M}:
\widetilde{\mathbb X}(L/K)^\times\rightarrow
\widetilde{\mathbb X}(M/K)^\times 
\end{equation}
defined by
\begin{equation}
\label{coleman-norm-def}
\widetilde{\mathcal N}_{L/M}
\left((\alpha_{\widetilde{E}_i})_{0\leq i\in\mathbb Z}\right)=
\left(\widetilde{N}_{E_i/E_i\cap
    M}(\alpha_{\widetilde{E}_i})\right)_{0\leq i\in\mathbb Z},
\end{equation}
for every $(\alpha_{\widetilde{E}_i})_{0\leq i\in\mathbb
  Z}\in\widetilde{\mathbb X}(L/K)^\times$.
\begin{remark}
\label{chain-independence}
The group homomorphism
\begin{equation*}
\widetilde{\mathcal N}_{L/M} : \widetilde{\mathbb X}(L/K)^\times
\rightarrow\widetilde{\mathbb X}(M/K)^\times
\end{equation*}
defined by eq.s (\ref{coleman-norm}) and (\ref{coleman-norm-def}) 
does \textit{not} depend on the choice of an ascending chain 
\begin{equation*}
K=E_o\subset E_1\subset\cdots\subset E_i\subset\cdots\subset L
\end{equation*}
satisfying $L=\bigcup_{0\leq i\in\mathbb Z} E_i$ and
$[E_{i+1}:E_i]<\infty$ for every $0\leq i\in\mathbb Z$.
\end{remark}
The basic properties of this group homomorphism are the following.
\begin{itemize}
\item[(i)]
If $U=(u_{\widetilde{E}_i})_{0\leq i\in\mathbb Z}
\in U_{\widetilde{\mathbb X}(L/K)}$, then
$\widetilde{\mathcal N}_{L/M}(U)\in U_{\widetilde{\mathbb X}(M/K)}$.
\begin{proof}
In fact, following the definition of the valuation 
$\nu_{\widetilde{\mathbb X}(M/K)}$ of $\widetilde{\mathbb X}(M/K)$ and
the definition of the valuation $\nu_{\widetilde{\mathbb X}(L/K)}$ of
$\widetilde{\mathbb X}(L/K)$, it follows that
\begin{equation*}
\begin{aligned}
\nu_{\widetilde{\mathbb X}(M/K)}
\left(\widetilde{\mathcal N}_{L/M}(U)\right) & =
\nu_{\widetilde{\mathbb X}(M/K)}
\left(\left(\widetilde{N}_{E_i/E_i\cap M}
(u_{\widetilde{E}_i})\right)_{0\leq i\in\mathbb Z}\right) \\
& = \nu_{\widetilde{K}}(u_{\widetilde{K}}) \\
& =0,
\end{aligned}
\end{equation*}
as 
\begin{equation*}
\nu_{\widetilde{\mathbb X}(L/K)}(U)=
\nu_{\widetilde{K}}(u_{\widetilde{K}})=0,
\end{equation*}
since $U\in U_{\widetilde{\mathbb X}(L/K)}$. 
\end{proof}
\item[(ii)]
If $U=(u_{\widetilde{E}_i})_{0\leq i\in\mathbb Z}
\in U_{\widetilde{\mathbb X}(L/K)}^\diamond$, 
then $\widetilde{\mathcal N}_{L/M}(U)
\in U_{\widetilde{\mathbb X}(M/K)}^\diamond$.
\begin{proof}
The assertion follows by observing that
$\text{Pr}_{\widetilde{K}}(U)=u_{\widetilde{K}}$ and 
$\text{Pr}_{\widetilde{K}}\left(\widetilde{\mathcal N}_{L/M}(U)\right)=
\widetilde{N}_{E_o/E_o\cap
  M}(u_{\widetilde{E}_o})=u_{\widetilde{K}}\in U_K$. 
\end{proof}
\item[(iii)]
If $U=(u_{E_i})_{0\leq i\in\mathbb Z}\in U_{{\mathbb X}(L/K)}$, then 
$\widetilde{\mathcal N}_{L/M}(U)\in U_{{\mathbb X}(M/K)}$.
\begin{proof}
The assertion follows by the definition eq. (\ref{coleman-norm-def}) of the 
homomorphism eq. (\ref{coleman-norm}) combined with the fact that 
$\widetilde{N}_{E_i/E_i\cap M}(u_{E_i})=N_{E_i/E_i\cap M}(u_{E_i})$
for every $u_{E_i}\in U_{E_i}$ and for every $0\leq i\in\mathbb Z$. 
\end{proof}
\end{itemize}
Thus, the group homomorphism eq. (\ref{coleman-norm}) 
defined by eq. (\ref{coleman-norm-def}) induces a group homomorphism,
which will be called the \textit{Coleman norm map from $L$ to $M$},
\begin{equation}
\label{coleman-norm-1}
\widetilde{\mathcal N}_{L/M}^{\text{Coleman}}:
U_{\widetilde{\mathbb X}(L/K)}^\diamond/U_{\mathbb X(L/K)}
\rightarrow U_{\widetilde{\mathbb X}(M/K)}^\diamond/U_{\mathbb X(M/K)}
\end{equation}
and defined by
\begin{equation}
\label{coleman-norm-1-def}
\widetilde{\mathcal N}_{L/M}^{\text{Coleman}}(\overline{U})=
\widetilde{\mathcal N}_{L/M}(U).U_{\mathbb X(M/K)},
\end{equation}
for every $U\in U_{\widetilde{\mathbb X}(L/K)}^\diamond$, where
$\overline{U}$ denotes, as before, the coset $U.U_{\mathbb X(L/K)}$ in
$U_{\widetilde{\mathbb X}(L/K)}^\diamond/U_{\mathbb X(L/K)}$.

The following theorem is stated, in the works of Fesenko \cite{fesenko-2000,
fesenko-2001, fesenko-2005},  without a proof. Thus, for the sake of 
completeness, we shall supply a proof of this theorem as well.
\begin{theorem}[Fesenko]
For the Galois sub-extension $M/K$ of $L/K$,
the square
\begin{equation}
\label{square-L/M-U}
\SelectTips{cm}{}\xymatrix{
{\text{Gal}(L/K)}\ar[r]^-{\phi_{L/K}^{(\varphi)}}\ar[d]_{\text{res}_M} & 
{U_{\widetilde{\mathbb X}{(L/K)}}^\diamond/U_{\mathbb X(L/K)}}
\ar[d]^{\widetilde{\mathcal N}_{L/M}^{\text{Coleman}}} \\
{\text{Gal}(M/K)}\ar[r]^-{\phi_{M/K}^{(\varphi)}} & 
{U_{\widetilde{\mathbb X}{(M/K)}}^\diamond/U_{\mathbb X(M/K)}},
}
\end{equation}
where the right-vertical arrow
\begin{equation*}
\widetilde{\mathcal N}_{L/M}^{\text{Coleman}} :
U_{\widetilde{\mathbb X}{(L/K)}}^\diamond/U_{\mathbb X(L/K)}
\rightarrow U_{\widetilde{\mathbb X}{(M/K)}}^\diamond/U_{\mathbb X(M/K)}
\end{equation*}
is the Coleman norm map from $L$ to $M$ defined by eq.s
(\ref{coleman-norm-1}) and (\ref{coleman-norm-1-def}), is commutative.
\end{theorem}
\begin{proof}
For each $\sigma\in\text{Gal}(L/K)$, we have to show that
\begin{equation*}
\widetilde{\mathcal N}_{L/M}^{\text{Coleman}}
\left(\phi_{L/K}^{(\varphi)}(\sigma)\right)=
\phi_{M/K}^{(\varphi)}(\sigma\mid_M) .
\end{equation*} 
Thus, it suffices to prove the congruence
\begin{equation*}
\widetilde{\mathcal N}_{L/M}(U_\sigma)\equiv
U_{\sigma\mid_M}\pmod{U_{\mathbb X(M/K)}},
\end{equation*}
or equivalently, it suffices to prove that
\begin{equation*}
\frac{\widetilde{\mathcal N}_{L/M}(U_\sigma)}
{\widetilde{\mathcal N}_{L/M}(U_\sigma)^\varphi}=
\frac{\Pi_{\varphi;M/K}^{\sigma\mid_M}}{\Pi_{\varphi;M/K}}.
\end{equation*}
Now, without loss of generality, in view of Remark
\ref{chain-independence}, the ascending chain of extensions
\begin{equation*}
K=E_o\subset E_1\subset\cdots\subset E_i\subset\cdots\subset L
\end{equation*}
can be chosen as the \textit{basic sequence} introduced in the beginning of
this section. Thus, each extension $E_i/K$ is finite and Galois for
$0\leq i\in\mathbb Z$.
Now, let $U_\sigma=(u_{\widetilde{E}_i})_{0\leq i\in\mathbb Z}
\in U_{\widetilde{\mathbb X}(L/K)}^\diamond$. Then, for each $0\leq
i\in\mathbb Z$,
\begin{equation*}
\frac{\widetilde{N}_{E_i/E_i\cap M}(u_{\widetilde{E}_i})}
{\widetilde{N}_{E_i/E_i\cap M}(u_{\widetilde{E}_i})^\varphi} =
\frac{\widetilde{N}_{E_i/E_i\cap M}(u_{\widetilde{E}_i})}
{\widetilde{N}_{E_i/E_i\cap M}(u_{\widetilde{E}_i}^\varphi)} =
\widetilde{N}_{E_i/E_i\cap M}\left(\frac{u_{\widetilde{E}_i}}
{u_{\widetilde{E}_i}^\varphi}\right) .
\end{equation*}
Now, the equality $\frac{u_{\widetilde{E}_i}}{u_{\widetilde{E}_i}^\varphi}
=\frac{\pi_{E_i}^\sigma}{\pi_{E_i}}$, which follows from 
the equation $\frac{U_\sigma}{U_\sigma^\varphi}=
\frac{\Pi_{\varphi;L/K}^\sigma}{\Pi_{\varphi;L/K}}$, yields
\begin{equation*}
\frac{\widetilde{N}_{E_i/E_i\cap M}(u_{\widetilde{E}_i})}
{\widetilde{N}_{E_i/E_i\cap M}(u_{\widetilde{E}_i})^\varphi} =
\widetilde{N}_{E_i/E_i\cap
  M}\left(\frac{\pi_{E_i}^\sigma}{\pi_{E_i}}\right) .
\end{equation*}
Thus, by the theorem of Koch and de Shalit, it follows that,
\begin{equation*}
\widetilde{N}_{E_i/E_i\cap M}
\left(\frac{\pi_{E_i}^\sigma}{\pi_{E_i}}\right) =
\frac{\widetilde{N}_{E_i/E_i\cap M}(\pi_{E_i})^\sigma}
{\widetilde{N}_{E_i/E_i\cap M}(\pi_{E_i})} =
\frac{\pi_{E_i\cap M}^{\sigma\mid_M}}{\pi_{E_i\cap M}},
\end{equation*}
which proves that
\begin{equation*}
\frac{\widetilde{\mathcal N}_{L/M}(U_\sigma)}
{\widetilde{\mathcal N}_{L/M}(U_\sigma)^\varphi}=
\frac{\Pi_{\varphi;M/K}^{\sigma\mid_M}}{\Pi_{\varphi;M/K}}.
\end{equation*} 
Now the proof is complete.
\end{proof}
Now, let $F/K$ be a finite sub-extension of $L/K$. 
Then, as $F$ is \textit{compatible} with $(K,\varphi)$, in the 
sense of \cite{koch-deshalit}
pp. 89, we may fix the Lubin-Tate splitting over $F$ to be
$\varphi_F=\varphi_K=\varphi$. 
Thus, there exists the chain of field extensions 
\begin{equation*}
K\subseteq F\subseteq L\subseteq K_\varphi\subseteq F_{\varphi},
\end{equation*}
where $L$ is a totally-ramified $APF$-Galois extension over $F$ by
Lemma \ref{apftower}. So there exists the mapping
\begin{equation*}
\phi_{L/F}^{(\varphi)}:\text{Gal}(L/F)\rightarrow
U_{\widetilde{\mathbb X}{(L/F)}}^\diamond/U_{\mathbb X(L/F)}
\end{equation*}
corresponding to the extension $L/F$.

For the $APF$-extension $L/F$, fix an ascending chain
\begin{equation*}
F=F_o\subset F_1\subset\cdots\subset F_i\subset\cdots\subset L
\end{equation*}
satisfying
$L=\bigcup_{0\leq i\in\mathbb Z}F_i$ and $[F_{i+1}:F_i]<\infty$ for
every $0\leq i\in\mathbb Z$. Introduce the homomorphism
\begin{equation}
\label{Lambda-map}
\Lambda_{F/K} : \widetilde{\mathbb X}(L/F)^\times\rightarrow
\widetilde{\mathbb X}(L/K)^\times
\end{equation}
by
\begin{equation}
\label{Lambda-map-definition}
\Lambda_{F/K} :
(\alpha_F\xleftarrow{\widetilde{N}_{F_1/F}}\alpha_{F_1}
\xleftarrow{\widetilde{N}_{F_2/F_1}}\cdots)\mapsto
(\widetilde{N}_{F/K}(\alpha_F)\xleftarrow{\widetilde{N}_{F/K}}
\alpha_F\xleftarrow{\widetilde{N}_{F_1/F}}\alpha_{F_1}
\xleftarrow{\widetilde{N}_{F_2/F_1}}\cdots),
\end{equation}
for each $\left(\alpha_{F_i}\right)_{0\leq i\in\mathbb Z}
\in\widetilde{\mathbb X}(L/F)^\times$. 
\begin{remark}
It is clear that, the homomorphism
\begin{equation*}
\Lambda_{F/K} : \widetilde{\mathbb X}(L/F)^\times\rightarrow
\widetilde{\mathbb X}(L/K)^\times
\end{equation*} 
defined by eq.s (\ref{Lambda-map}) and 
(\ref{Lambda-map-definition}) does \textit{not} depend on
the choice of ascending chain of fields 
\begin{equation*}
F=F_o\subset F_1\subset\cdots\subset F_i\subset\cdots\subset L
\end{equation*}
satisfying
$L=\bigcup_{0\leq i\in\mathbb Z}F_i$ and $[F_{i+1}:F_i]<\infty$ for
every $0\leq i\in\mathbb Z$.
\end{remark}
The basic properties of this group homomorphism are the following
\begin{itemize}
\item[(i)]
The square
\begin{equation*}
\SelectTips{cm}{}\xymatrix{
{\widetilde{\mathbb X}(L/F)}\ar[r]^-{\Lambda_{F/K}} & 
{\widetilde{\mathbb X}(L/K)} \\
{\mathbb X(L/F)}\ar[r]^-{\Lambda_{F/K}}\ar[u]^{\text{inc.}} & 
{\mathbb X(L/K)}\ar[u]_{\text{inc.}}
}
\end{equation*}
is commutative.
\item[(ii)] 
If $U=(u_{\widetilde{F}_i})_{0\leq i\in\mathbb Z}\in
  U_{\widetilde{\mathbb X}(L/F)}$, then $\Lambda_{F/K}(U)\in
  U_{\widetilde{\mathbb X}(L/K)}$.
\item[(iii)]
If $U=(u_{\widetilde{F}_i})_{0\leq i\in\mathbb Z}\in
  U_{\widetilde{\mathbb X}(L/F)}^\diamond$, then $\Lambda_{F/K}(U)\in
  U_{\widetilde{\mathbb X}(L/K)}^\diamond$.
\item[(iv)]
If $U=(u_{F_i})_{0\leq i\in\mathbb Z}\in
  U_{\mathbb X(L/F)}$, then $\Lambda_{F/K}(U)\in
  U_{\mathbb X(L/K)}$.
\end{itemize}
Thus, the group homomorphism eq. (\ref{Lambda-map}) defined by
eq. (\ref{Lambda-map-definition}) induces a group homomorphism
\begin{equation}
\label{lambda-map}
\lambda_{F/K} : 
U_{\widetilde{\mathbb X}{(L/F)}}^\diamond/U_{\mathbb X(L/F)}
\rightarrow
U_{\widetilde{\mathbb X}{(L/K)}}^\diamond/U_{\mathbb X(L/K)}
\end{equation}
defined by
\begin{equation}
\label{lambda-map-definition}
\lambda_{F/K}:\overline{U}\mapsto\Lambda_{F/K}(U).U_{\mathbb X(L/K)},
\end{equation}
for every $U\in U_{\widetilde{\mathbb X}{(L/F)}}^\diamond$, where
$\overline{U}$ denotes, as before, the coset $U.U_{\mathbb X(L/F)}$ 
in $U_{\widetilde{\mathbb X}{(L/F)}}^\diamond/U_{\mathbb X(L/F)}$.

The following theorem is stated, in the works of Fesenko 
\cite{fesenko-2000, fesenko-2001, fesenko-2005},  without a proof. 
Thus, for the sake of completeness, we shall supply a proof of this 
theorem as well.
\begin{theorem}[Fesenko]
For the finite sub-extension $F/K$ of $L/K$, the square
\begin{equation}
\label{square-F/K-U}
\SelectTips{cm}{}\xymatrix{
{\text{Gal}(L/F)}\ar[r]^-{\phi_{L/F}^{(\varphi)}}\ar[d]_{\text{inc.}} & 
{U_{\widetilde{\mathbb X}{(L/F)}}^\diamond/U_{\mathbb X(L/F)}}
\ar[d]^{\lambda_{F/K}} \\
{\text{Gal}(L/K)}\ar[r]^-{\phi_{L/K}^{(\varphi)}} & 
{U_{\widetilde{\mathbb X}{(L/K)}}^\diamond/U_{\mathbb X(L/K)}},
}
\end{equation}
where the right-vertical arrow
\begin{equation*}
\lambda_{F/K} : U_{\widetilde{\mathbb X}{(L/F)}}^\diamond/U_{\mathbb X(L/F)}
\rightarrow U_{\widetilde{\mathbb X}{(L/K)}}^\diamond/U_{\mathbb X(L/K)}
\end{equation*}
is defined by eq.s (\ref{lambda-map}) and
(\ref{lambda-map-definition}), 
is commutative.
\end{theorem}
\begin{proof}
For $\sigma\in\text{Gal}(L/F)$,
$\phi_{L/F}^{(\varphi)}(\sigma)=U_\sigma .U_{\mathbb X(L/F)}$, where
$U_\sigma\in U_{\widetilde{\mathbb X}(L/F)}^\diamond$ satisfies the
equality
\begin{equation}
\label{phi-definition}
\frac{U_\sigma}{U_\sigma^\varphi}=\frac{\Pi_{\varphi;L/F}^\sigma}
{\Pi_{\varphi;L/F}} .
\end{equation}
Here, $\Pi_{\varphi;L/F}$ is the norm compatible sequence of primes
$(\pi_{F_i})_{0\leq i\in\mathbb Z}$. Now,
\begin{equation*}
\Lambda_{F/K}\left(\frac{U_\sigma}{U_\sigma^\varphi}\right) =
\frac{\Lambda_{F/K}(U_\sigma)}{\Lambda_{F/K}(U_\sigma^\varphi)} =
\frac{\Lambda_{F/K}(U_\sigma)}{\Lambda_{F/K}(U_\sigma)^\varphi} .
\end{equation*}
On the other hand,
$\Lambda_{F/K}(\Pi_{\varphi;L/F})=\Pi_{\varphi;L/K}$ and 
$\Lambda_{F/K}(\Pi_{\varphi;L/F}^\sigma)=\Pi_{\varphi;L/K}^\sigma$.
Thus, eq. (\ref{phi-definition}) yields
\begin{equation*}
\frac{\Lambda_{F/K}(U_\sigma)}{\Lambda_{F/K}(U_\sigma)^\varphi}=
\frac{\Pi_{\varphi;L/K}^\sigma}
{\Pi_{\varphi;L/K}},
\end{equation*} 
which shows that
\begin{equation*}
\phi_{L/K}^{(\varphi)}(\sigma)=\Lambda_{F/K}(U_\sigma).U_{\mathbb
  X(L/K)}= \lambda_{F/K}(\phi_{L/F}^{(\varphi)}(\sigma)),
\end{equation*}
completing the proof of the commutativity of the square.
\end{proof}
If $L/K$ is furthermore a finite extension, then the composition
\begin{equation*}
\SelectTips{cm}{}\xymatrix{
{\text{Gal}(L/K)}\ar[r]^-{\phi_{L/K}^{(\varphi)}}\ar@/^3pc/[rr]^{\iota_{L/K}} &
U_{\widetilde{\mathbb X}{(L/K)}}^\diamond/U_{\mathbb X(L/K)}
\ar[r]^-{\text{Pr}_{\widetilde{K}}} & 
U_K/N_{L/K}U_L, 
}
\end{equation*}
is the Iwasawa-Neukirch map of the extension $L/K$. Thus,
\textit{the mapping $\phi_{L/K}^{(\varphi)}$ defined for $L/K$ is the 
generalization of the Iwasawa-Neukirch map 
$\iota_{L/K} : \text{Gal}(L/K)\rightarrow U_K/N_{L/K}(U_L)$
for the totally-ramified $APF$-Galois sub-extensions $L/K$ of 
$K_\varphi/K$}. 

Likewise, we can extend the definition of the Hazewinkel map
$h_{L/K}:U_K/N_{L/K}U_L\rightarrow\text{Gal}(L/K)^{ab}$ initially
defined for totally-ramified finite Galois extensions $L/K$ to
totally-ramified $APF$-Galois sub-extensions of $K_\varphi/K$ by 
generalizing Serre short exact sequence introduced in eq.s (\ref{serre1}) 
and (\ref{serre2}).
In order to do so, we first have to assume that the local field $K$
satisfies the condition
\begin{equation}
\label{rootofunity}
\pmb{\mu}_p(K^{sep})=\{\alpha\in K^{sep}:\alpha^p=1\}\subset K,
\end{equation}
where $p=\text{char}(\kappa_K)$. 
\begin{remark}
If $K$ is a local field of characteristic $p=\text{char}(\kappa_K)$, 
the assumption (\ref{rootofunity}) on $K$ is automatically satisfied.
For details on the assumption (\ref{rootofunity}) on $K$, we refer the reader
to \cite{fesenko-2000, fesenko-2001, fesenko-2005}.   
\end{remark}
In what follows,
as before, let $L/K$ be a totally-ramified $APF$-Galois extension
satisfying eq. (\ref{phitower}).
Under this assumption,
there exists a topological $\text{Gal}(L/K)$-submodule $Y_{L/K}$ of
$U_{\widetilde{\mathbb X}{(L/K)}}^\diamond$, such that 
\begin{itemize}
\item[(i)]
$U_{\mathbb X(L/K)}\subseteq Y_{L/K}$; 
\item[(ii)]
the composition
\begin{equation*}
\Phi_{L/K}^{(\varphi)}:\text{Gal}(L/K)\xrightarrow{\phi_{L/K}^{(\varphi)}} 
U_{\widetilde{\mathbb X}{(L/K)}}^\diamond/U_{\mathbb X(L/K)}
\xrightarrow[\txt{canonical \\topol. map}]{c_{L/K}}
U_{\widetilde{\mathbb X}{(L/K)}}^\diamond/Y_{L/K}
\end{equation*}
is a bijection with the extended Hazewinkel map 
$H_{L/K}^{(\varphi)}:
U_{\widetilde{\mathbb X}{(L/K)}}^\diamond/Y_{L/K}\rightarrow
\text{Gal}(L/K)$
as the inverse.
\end{itemize}
Now, we shall briefly review the constructions of the topological group 
$Y_{L/K}$ and the extended Hazewinkel map
$H_{L/K}^{(\varphi)}:
U_{\widetilde{\mathbb X}{(L/K)}}^\diamond/Y_{L/K}\rightarrow
\text{Gal}(L/K)$. For details, we refer the reader to 
\cite{fesenko-2000, fesenko-2001, fesenko-2005}, which we follow closely. 

Fix a \textit{basic} ascending chain of sub-extensions in $L/K$
\begin{equation}
\label{basic-ascending-chain-L/K}
K=K_o\subset K_1\subset\cdots\subset K_i\subset\cdots\subset L .
\end{equation} 
\textit{once and for all}.
Now, introduce the following notation. For each $1\leq i\in\mathbb Z$,
\begin{itemize}
\item[(i)]
let $\sigma_i$ denote an element of
$\text{Gal}(\widetilde{L}/\widetilde{K})$ satisfying
$<\sigma_i\mid_{K_i}>=\text{Gal}(K_i/K_{i-1})$;
\item[(ii)]
let $\widetilde{K}_i=K_i\widetilde{K}$.
\end{itemize}
By abelian local class field theory, for each $1\leq k\in\mathbb Z$, 
there exists an injective homomorphism
\begin{equation}
\label{injective-map}
\Xi_{K_{k+1}/K_k} : \text{Gal}(K_{k+1}/K_k)\rightarrow 
U_{\widetilde{K}_{k+1}}/U_{\widetilde{K}_{k+1}}^{\sigma_{k+1}-1}
\end{equation}
defined by
\begin{equation}
\label{injective-map-definition}
\Xi_{K_{k+1}/K_k} :
\tau\mapsto\pi_{K_{k+1}}^{\tau-1} U_{\widetilde{K}_{k+1}}^{\sigma_{k+1}-1},
\end{equation}
for every $\tau\in\text{Gal}(K_{k+1}/K_k)$. 
Let $\text{im}(\Xi_{K_{k+1}/K_k})=T_k^{(L/K)}=T_k$ be the
isomorphic copy of 
$\text{Gal}(K_{k+1}/K_k)$ in 
$U_{\widetilde{K}_{k+1}}/U_{\widetilde{K}_{k+1}}^{\sigma_{k+1}-1}$.
\begin{theorem}[Fesenko]
\label{split-exact-sequence}
Fix $1\leq k\in\mathbb Z$. Let
\begin{equation*}
T_k^{(L/K)'}=T_k'=T_k\cap\left(\prod_{1\leq i\leq k+1}U_{\widetilde{K}_{k+1}}
^{\sigma_i-1}\right)/U_{\widetilde{K}_{k+1}}^{\sigma_{k+1}-1}. 
\end{equation*}
Then the exact sequence
\begin{equation*}
\SelectTips{cm}{}\xymatrix@1{
1\ar[r] & T_k'\ar[r] & 
{\left(\prod_{1\leq i\leq k+1}U_{\widetilde{K}_{k+1}}
^{\sigma_i-1}\right)/U_{\widetilde{K}_{k+1}}^{\sigma_{k+1}-1}}
\ar[r]^-{\widetilde{N}_{K_{k+1}/K_k}} &
{\prod_{1\leq i\leq k}U_{\widetilde{K}_k}^{\sigma_i-1}}
\ar[r]\ar@/_2.2pc/[l]_{h_k^{(L/K)}=h_k}&1}
\end{equation*}
splits by a homomorphism
\begin{equation*}
h_k :\prod_{1\leq i\leq k}U_{\widetilde{K}_k}^{\sigma_i-1}\rightarrow
\left(\prod_{1\leq i\leq k+1}U_{\widetilde{K}_{k+1}}^{\sigma_i-1}\right)/
U_{\widetilde{K}_{k+1}}^{\sigma_{k+1}-1} .
\end{equation*}
This homomorphism is not unique in general.
\end{theorem}
For each $1\leq k\in\mathbb Z$, consider any map
\begin{equation*}
g_k^{(L/K)}=g_k :\prod_{1\leq i\leq k}U_{\widetilde{K}_k}^{\sigma_i-1}
\rightarrow
\prod_{1\leq i\leq k+1}U_{\widetilde{K}_{k+1}}^{\sigma_i-1}
\end{equation*}
which makes the triangle
\begin{equation*}
\SelectTips{cm}{}\xymatrix{
{} & {\prod_{1\leq i\leq k+1}U_{\widetilde{K}_{k+1}}^{\sigma_i-1}}
\ar[dd]^{\pmod{U_{\widetilde{K}_{k+1}}^{\sigma_{k+1}-1}}} \\
 &  &  \\
{\prod_{1\leq i\leq k}U_{\widetilde{K}_k}^{\sigma_i-1}}
\ar[r]^-{h_k}\ar[uur]^{g_k} & 
{\left(\prod_{1\leq i\leq k+1}U_{\widetilde{K}_{k+1}}^{\sigma_i-1}\right)/
U_{\widetilde{K}_{k+1}}^{\sigma_{k+1}-1}}
}
\end{equation*}
commutative. Such a map clearly exists.
Now, choose, for every $1\leq i\in\mathbb Z$, a mapping 
\begin{equation*}
f_i^{(L/K)}=f_i : U_{\widetilde{K}_i}^{\sigma_i-1}\rightarrow
U_{\widetilde{\mathbb X}(L/K_i)}\xrightarrow{\Lambda_{K_i/K}}
U_{\widetilde{\mathbb X}(L/K)}
\end{equation*}
which satisfies for each $j\in\mathbb Z_{>i}$ the equality
\begin{equation*}
\text{Pr}_{\widetilde{K}_j}\circ f_i=
(g_{j-1}\circ\cdots\circ g_i)\mid_{U_{\widetilde{K}_i}^{\sigma_i-1}},
\end{equation*}
where $\text{Pr}_{\widetilde{K}_j}:U_{\widetilde{\mathbb X}(L/K)}\rightarrow
U_{\widetilde{K}_j}$ denotes the projection on the
$\widetilde{K}_j$-coordinate. 
\begin{lemma}[Fesenko]
\label{infinite-product}
\begin{itemize}
\item[(i)]
Let $z^{(i)}\in\text{im}(f_i)=Z_i^{(L/K)}$ for each $1\leq i\in\mathbb Z$. 
Then the infinite product $\prod_i z^{(i)}$ converges to an element $z$ in
$U_{\widetilde{\mathbb X}(L/K)}^\diamond$. 
\item[(ii)]
Let
\begin{equation*}
Z_{L/K}(\{K_i,f_i\})=
\left\{\prod_{1\leq i\in\mathbb Z} z^{(i)} : 
z^{(i)}\in\text{im}(f_i)\right\}.
\end{equation*}
Then, $Z_{L/K}(\{K_i,f_i\})$ is a topological subgroup of 
$U_{\widetilde{\mathbb X}(L/K)}^\diamond$.
\end{itemize}
\end{lemma}
\begin{remark}
\label{Z-subset-U1}
In fact, $Z_{L/K}(\{K_i,f_i\})$ is a topological subgroup of
$U_{\widetilde{\mathbb X}(L/K)}^1$. Let $z\in Z_{L/K}(\{K_i,f_i\})$ and
choose $z^{(i)}\in\text{im}(f_i)\subset U_{\widetilde{\mathbb X}(L/K)}$ 
so that $z=\prod_i z^{(i)}$. 
It suffices to show that $\text{Pr}_{\widetilde K}(z^{(i)})=1_K$. In order to 
do so, let 
$\alpha^{(i)}\in U_{\widetilde{K}_i}^{\sigma_i-1}$ such that
$f_i(\alpha^{(i)})=z^{(i)}$.
Thus, $\text{Pr}_{\widetilde{K}_i}(z^{(i)})=\alpha^{(i)}$. 
Now, by Hilbert $90$,
it follows that, $\widetilde N_{K_i/K}(\alpha^{(i)})
=(\widetilde N_{K_{i-1}/K}\circ\widetilde N_{K_i/K_{i-1}})(\alpha^{(i)})
=1_K$, which completes the proof.
\end{remark}
\begin{lemma}
\label{gal-action-Z}
For $1\leq i\in\mathbb Z$, let $\sigma=\sigma_i\in
\text{Gal}(\widetilde{L}/\widetilde{K})$ such that
$<\sigma\mid_{K_i}>=\text{Gal}(K_i/K_{i-1})$.
Let $\tau\in\text{Gal}(L/K)$ viewed as an element of 
$\text{Gal}(\widetilde{L}/\widetilde{K})$. Then
\begin{equation*}
\left(U_{\widetilde{K}_i}^{\sigma-1}\right)^\tau=
U_{\widetilde{K}_i}^{\sigma-1}.
\end{equation*}
\end{lemma}
\begin{proof}
Let $\tau$ be any element of $\text{Gal}(L/K)$. Now consider $\tau$ as an
element of $\text{Gal}(\widetilde{L}/\widetilde{K})$. Then clearly,
the conjugate $\tau^{-1}\sigma\tau\in\text{Gal}(\widetilde{L}/\widetilde{K})$
satisfies $<\tau^{-1}\sigma\tau\mid_{K_i}>=\text{Gal}(K_i/K_{i-1})$ as
$\left(\tau^{-1}\sigma\tau\mid_{K_i}\right)^n=\text{id}_{K_i}$ yields
$\left(\sigma\mid_{K_i}\right)^n=\text{id}_{K_i}$. Let $0<d\in\mathbb Z$
such that $\tau^{-1}\sigma\tau\mid_{K_i}=(\sigma\mid_{K_i})^d
=(\sigma^d)\mid_{K_i}$. Thus $\tau^{-1}\sigma\tau\sigma^{-d}\in
\text{Gal}(\widetilde{L}/\widetilde{K}_i)$ as $\widetilde{K}_i
=\widetilde{K}K_i$. It then follows that 
\begin{equation*}
U_{\widetilde{K}_i}^{\tau^{-1}\sigma\tau-1}=U_{\widetilde{K}_i}^{\sigma^d-1}.
\end{equation*}
As $U_{\widetilde{K}_i}^{\tau^{-1}\sigma\tau-1}
=U_{\widetilde{K}_i}^{\tau^{-1}(\sigma-1)\tau}
=\left(U_{\widetilde{K}_i}^{\sigma-1}\right)^\tau$, the equality
\begin{equation*}
\left(U_{\widetilde{K}_i}^{\sigma-1}\right)^\tau
=U_{\widetilde{K}_i}^{\sigma^d-1}
\end{equation*}
follows as well. Now, the inclusion
\begin{equation*}
U_{\widetilde{K}_i}^{\sigma^d-1}\subseteq
U_{\widetilde{K}_i}^{\sigma-1}
\end{equation*}
is clear, because, for $u\in U_{\widetilde{K}_i}$,
\begin{equation*}
\frac{u^{\sigma^d}}{u}=\frac{\left(u^{\sigma^{d-1}}\right)^\sigma}
{u^{\sigma^{d-1}}}\cdots
\frac{\left(u^{\sigma}\right)^\sigma}{u^\sigma}\frac{u^\sigma}{u}.
\end{equation*}
Thus, for $\tau\in\text{Gal}(L/K)$, the inclusion
$\left(U_{\widetilde{K}_i}^{\sigma-1}\right)^\tau\subseteq 
U_{\widetilde{K}_i}^{\sigma-1}$ follows. Hence, for $\tau\in\text{Gal}(L/K)$,
\begin{equation*}
\left(U_{\widetilde{K}_i}^{\sigma-1}\right)^\tau = 
U_{\widetilde{K}_i}^{\sigma-1}
\end{equation*}
completing the proof.
\end{proof}
Now, let $\tau\in\text{Gal}(L/K)$. Consider the element $\tau^{-1}\sigma_i\tau$
of $\text{Gal}(\widetilde{L}/\widetilde{K})$ for each $1\leq i\in\mathbb Z$.
Then clearly, $<\tau^{-1}\sigma_i\tau\mid_{K_i}>=\text{Gal}(K_i/K_{i-1})$.
By abelian local class field theory and by Lemma \ref{gal-action-Z}, the 
square
\begin{equation*}
\SelectTips{cm}{}\xymatrix{
{\text{Gal}(K_i/K_{i-1})}\ar[rr]^{\Xi_{K_i/K_{i-1}}}
\ar[d]_{\text{$\tau$-conjugation}}  & & 
{U_{\widetilde{K}_i}/U_{\widetilde{K}_i}^{\sigma_i-1}}\ar[d]^{\tau} \\
{\text{Gal}(K_i/K_{i-1})}\ar[rr]^{\Xi_{K_i/K_{i-1}}} & & {U_{\widetilde{K}_i}/
U_{\widetilde{K}_i}^{\sigma_i-1}}
}
\end{equation*}
is commutative, where the $\tau$-conjugation map $\text{Gal}(K_i/K_{i-1})
\rightarrow\text{Gal}(K_i/K_{i-1})$ is defined by 
$\gamma\mapsto\tau^{-1}\gamma\tau$ for every 
$\gamma\in\text{Gal}(K_i/K_{i-1})$. 
Thus, it follows that 
\begin{equation*}
\text{im}\left(\Xi_{K_i/K_{i-1}}\right)^\tau
=\text{im}\left(\Xi_{K_i/K_{i-1}}\right).
\end{equation*}
Now, following Theorem \ref{split-exact-sequence}, for
\begin{equation*}
T_i^\tau=T_i=\text{im}(\Xi_{K_{i+1}/K_i})
\end{equation*}
and
\begin{equation*}
(T_i')^\tau=T_i^\tau\cap\left(\prod_{1\leq j\leq i+1}U_{\widetilde{K}_{i+1}}
^{\tau^{-1}\sigma_j\tau-1}\right)/U_{\widetilde{K}_{i+1}}
^{\tau^{-1}\sigma_{i+1}\tau-1},
\end{equation*}
the exact sequence
\begin{equation*}
\SelectTips{cm}{}\xymatrix@1{
1\ar[r] & (T_i')^\tau \ar[r] & 
{\left(\prod_{1\leq j\leq i+1}U_{\widetilde{K}_{i+1}}
^{\tau^{-1}\sigma_j\tau-1}\right)/U_{\widetilde{K}_{i+1}}^{\tau^{-1}
\sigma_{i+1}\tau-1}}
\ar[r]^-{\widetilde{N}_{K_{i+1}/K_i}} &
{\prod_{1\leq j\leq i}U_{\widetilde{K}_i}^{\tau^{-1}\sigma_j\tau-1}}
\ar[r]\ar@/_2.2pc/[l]_{(h_i^{(L/K)})^\tau=h_i^\tau}&1}
\end{equation*}
splits by a homomorphism
\begin{equation*}
h_i^\tau :\prod_{1\leq j\leq i}U_{\widetilde{K}_i}^{\tau^{-1}\sigma_j\tau-1}
\rightarrow
\left(\prod_{1\leq j\leq i+1}U_{\widetilde{K}_{i+1}}^{\tau^{-1}\sigma_j\tau-1}
\right)/U_{\widetilde{K}_{i+1}}^{\tau^{-1}\sigma_{i+1}\tau-1},
\end{equation*}
and which furthermore makes the diagram 
\begin{equation*}
\SelectTips{cm}{}\xymatrix{
1\ar[r] & T_i' \ar[r]\ar[d]_{\tau} & 
{\left(\prod_{1\leq j\leq i+1}U_{\widetilde{K}_{i+1}}
^{\sigma_j-1}\right)/U_{\widetilde{K}_{i+1}}^{\sigma_{i+1}-1}}
\ar[r]^-{\widetilde{N}_{K_{i+1}/K_i}}\ar[d]_{\tau} &
{\prod_{1\leq j\leq i}U_{\widetilde{K}_i}^{\sigma_j-1}}
\ar[r]\ar@/_2.2pc/[l]_{h_i^{(L/K)}=h_i}\ar[d]_{\tau} & 1 \\
1\ar[r] & (T_i')^\tau \ar[r] & 
{\left(\prod_{1\leq j\leq i+1}U_{\widetilde{K}_{i+1}}
^{\tau^{-1}\sigma_j\tau-1}\right)/U_{\widetilde{K}_{i+1}}^{\tau^{-1}
\sigma_{i+1}\tau-1}}
\ar[r]^-{\widetilde{N}_{K_{i+1}/K_i}} &
{\prod_{1\leq j\leq i}U_{\widetilde{K}_i}^{\tau^{-1}\sigma_j\tau-1}}
\ar[r]\ar@/^2.2pc/[l]^{(h_i^{(L/K)})^\tau=h_i^\tau} & 1
}
\end{equation*}
commutative. Thus, it follows that, there exists a map
\begin{equation*}
g_i^\tau :\prod_{1\leq j\leq i}U_{\widetilde{K}_i}^{\tau^{-1}\sigma_j\tau-1}
\rightarrow\prod_{1\leq j\leq i+1}U_{\widetilde{K}_{i+1}}
^{\tau^{-1}\sigma_j\tau-1}
\end{equation*}
which makes the following diagram
{\scriptsize
\begin{equation*}
\SelectTips{cm}{}\xymatrix{
{} & {} & {\prod_{1\leq j\leq i+1}U_{\widetilde{K}_{i+1}}
^{\tau^{-1}\sigma_j\tau-1}}\ar[dd]^{\pmod{U_{\widetilde{K}_{i+1}}
^{\sigma_{i+1}-1}}} & {} \\
{} & {} & {} & {\prod_{1\leq j\leq i+1}U_{\widetilde{K}_{i+1}}^{\sigma_j-1}}
\ar[dd]^{\pmod{U_{\widetilde{K}_{i+1}}^{\sigma_{i+1}-1}}}
\ar@/_2.2pc/[ul]_\tau \\
{\prod_{1\leq j\leq i}U_{\widetilde{K}_i}^{\tau^{-1}\sigma_j\tau-1}}
\ar[rr]^{h_i^\tau}\ar@/^2.2pc/[uurr]^{g_i^\tau} & {} & 
{\left(\prod_{1\leq j\leq i+1}
U_{\widetilde{K}_{i+1}}^{\tau^{-1}\sigma_j\tau-1}\right)/
U_{\widetilde{K}_{i+1}}^{\sigma_{i+1}-1}} & {} \\
{} & {\prod_{1\leq j\leq i}U_{\widetilde{K}_i}^{\sigma_j-1}}
\ar[rr]^{h_i}\ar@/_2.2pc/[uurr]_{g_i}\ar@/_2.2pc/[ul]_-{\tau} & {} & 
{\left(\prod_{1\leq j\leq i+1}U_{\widetilde{K}_{i+1}}^{\sigma_j-1}\right)/
U_{\widetilde{K}_{i+1}}^{\sigma_{i+1}-1}}\ar@/_2.2pc/[ul]_-{\tau}
}
\end{equation*}
}
commutative. Now, for every $1\leq i\in\mathbb Z$, choose a mapping
\begin{equation*}
f_i^\tau :U_{\widetilde{K}_i}^{\sigma_i-1}\rightarrow 
U_{\widetilde{\mathbb X}(L/K)}
\end{equation*}
which satisfies for each $j\in\mathbb Z_{>i}$ the equality
\begin{equation*}
\text{Pr}_{\widetilde{K}_j}\circ f_i^\tau=(g_{j-1}^\tau\circ\cdots
\circ g_i^\tau)\mid_{U_{\widetilde{K}_i}^{\sigma_i-1}}.
\end{equation*}
Thus, for $j\in\mathbb Z_{>i}$, and for 
$\alpha\in U_{\widetilde{K}_i}^{\sigma_i-1}$,
\begin{equation*}
\text{Pr}_{\widetilde{K}_j}\circ f_i^\tau(\alpha)=
\left(\text{Pr}_{\widetilde{K}_j}\circ f_i(\alpha^{\tau^{-1}})\right)^\tau ,
\end{equation*}
which yields the following relationship
\begin{equation}
\label{tau-relationship}
f_i^\tau(\alpha)=f_i\left(\alpha^{\tau^{-1}}\right)^\tau ,
\end{equation}
for every $\alpha\in U_{\widetilde{K}_i}^{\sigma_i-1}$.
After all these observations, an immediate consequence of Lemma 
\ref{gal-action-Z} is the following corollary.
\begin{corollary}
\label{gal-module-Z}
For $\tau\in\text{Gal}(L/K)$, 
\begin{equation*}
Z_{L/K}(\{K_i,f_i\})^\tau =Z_{L/K}(\{K_i,f_i^\tau\}).
\end{equation*}
\end{corollary}
\begin{proof}
Let $z\in Z_{L/K}(\{K_i,f_i\})$ and choose $z^{(i)}\in\text{im}(f_i)
\subset U_{\widetilde{\mathbb X}(L/K)}$ such that $z=\prod_iz^{(i)}$.
By the continuity of the action of $\text{Gal}(L/K)$ on 
$U_{\widetilde{\mathbb X}(L/K)}$, to prove that 
$z^\tau\in Z_{L/K}(\{K_i,f_i^\tau\})$, it suffices to show that 
$\left(z^{(i)}\right)^\tau\in\text{im}(f_i^\tau)$.
Now, let $\alpha^{(i)}\in U_{\widetilde{K}_i}^{\sigma_i-1}$ such that
$f_i(\alpha^{(i)})=z^{(i)}$. Then, 
$\left(z^{(i)}\right)^\tau=f_i\left(\alpha^{(i)}\right)^\tau
=f_i\left(((\alpha^{(i)})^\tau)^{\tau^{-1}}\right)^\tau
=f_i^\tau\left((\alpha^{(i)})^\tau\right)$ by eq. (\ref{tau-relationship}), 
where by the previous Lemma \ref{gal-action-Z},
$\left(\alpha^{(i)}\right)^\tau\in U_{\widetilde{K}_i}^{\sigma_i-1}$. Thus,
$\left(z^{(i)}\right)^\tau\in\text{im}(f_i^\tau)$.
\end{proof}
\begin{remark}
\label{remark-gal-module-Z}
By pp. 71 of \cite{fesenko-2001}, for $\tau\in\text{Gal}(L/K)$,
$Z_{L/K}\left(\{K_i,f_i\}\right)$ and
$Z_{L/K}\left(\{K_i,f_i^\tau\}\right)$ are algebraically and topologically
isomorphic. Thus, Corollary
\ref{gal-module-Z} indeed defines a continuous action of $\text{Gal}(L/K)$ on
$Z_{L/K}\left(\{K_i,f_i\}\right)$.
\end{remark}
Now, define the topological subgroup $Y_{L/K}\left(\{K_i,f_i\}\right)=Y_{L/K}$ 
of $U_{\widetilde{\mathbb X}(L/K)}^\diamond$ to be
\begin{equation}
\label{definition-of-Y}
Y_{L/K}=
\left\{y\in U_{\widetilde{\mathbb X}(L/K)} : y^{1-\varphi}
\in Z_{L/K}(\{K_i,f_i\})\right\}.
\end{equation}
\begin{lemma}
$Y_{L/K}$ is a topological $\text{Gal}(L/K)$-submodule of
$U_{\widetilde{\mathbb X}(L/K)}^\diamond$.
\end{lemma}
\begin{proof}
Let $\tau\in\text{Gal}(L/K)$ and $y\in Y_{L/K}$. Note that,
$(y^\tau)^\varphi=(y^\varphi)^\tau$, as the action of $\tau$ on 
$y=(u_{\widetilde{K}_i})_{0\leq i\in\mathbb Z}$ is defined by
the action of $\tau$ on the ``$K_i$-part'' of $u_{\widetilde{K}_i}$
for each $0\leq i\in\mathbb Z$, and the action of $\varphi$ on
$y=(u_{\widetilde{K}_i})_{0\leq i\in\mathbb Z}$ is defined by
the action of $\varphi$ on the ``$\widetilde{K}$-part'' of 
$u_{\widetilde{K}_i}$ for each $0\leq i\in\mathbb Z$. Thus,
$\frac{y^\tau}{(y^\tau)^\varphi}=\frac{y^\tau}{(y^\varphi)^\tau}=
\left(\frac{y}{y^\varphi}\right)^\tau\in Z_{L/K}(\{K_i,f_i\})^\tau$.
Now, the proof follows from Corollary \ref{gal-module-Z} and by Remark
\ref{remark-gal-module-Z}.
\end{proof}
\begin{lemma}[Fesenko]
\label{non-abelian-hazewinkel-1}
The mapping
\begin{equation*}
\ell_{L/K}^{(\varphi)}:\text{Gal}(L/K)\rightarrow
U_{\widetilde{\mathbb X}(L/K)}^1/Z_{L/K}(\{K_i,f_i\})
\end{equation*}
defined by
\begin{equation*}
\ell_{L/K}^{(\varphi)} : \sigma\mapsto \Pi_{\varphi ; L/K}^{\sigma -1}.
Z_{L/K}(\{K_i,f_i\}),
\end{equation*}
for every $\sigma\in\text{Gal}(L/K)$, is a group isomorphism, where 
the group operation $\ast$ on
$U_{\widetilde{\mathbb X}(L/K)}^1/Z_{L/K}(\{K_i,f_i\})$ is defined by 
\begin{equation*}
\overline{U}\ast\overline{V}=\overline{U}.
\overline{V}^{(\ell_{L/K}^{(\varphi)})^{-1}(\overline{U})}
\end{equation*}
for every $\overline{U}=U.Z_{L/K}(\{K_i,f_i\}), 
\overline{V}=V.Z_{L/K}(\{K_i,f_i\})\in 
U_{\widetilde{\mathbb X}{(L/K)}}^1/Z_{L/K}(\{K_i,f_i\})$ with $U,V\in
U_{\widetilde{\mathbb X}{(L/K)}}^1$. 
\end{lemma}
Now, introduce the \textit{fundamental exact sequence} as
\begin{equation*}
1\rightarrow\text{Gal}(L/K)\xrightarrow{\ell_{L/K}^{(\varphi)}}
U_{\widetilde{\mathbb X}(L/K)}/Z_{L/K}(\{K_i,f_i\})
\xrightarrow{\text{Pr}_{\widetilde K}} U_{\widetilde K}\rightarrow 1
\end{equation*}
as a generalization of Serre short exact sequence 
(cf. eq.s (\ref{serre1}) and (\ref{serre2})). Thus, for any 
$U\in U_{\widetilde{\mathbb X}{(L/K)}}^\diamond$, as 
$U^{1-\varphi}\in U_{\widetilde{\mathbb X}{(L/K)}}^1$, there exists a 
unique $\sigma_U\in\text{Gal}(L/K)$ satisfying
\begin{equation}
\label{non-abelian-hazewinkel-3/2}
U^{1-\varphi}.Z_{L/K}(\{K_i,f_i\})=\ell_{L/K}^{(\varphi)}(\sigma_U),
\end{equation} 
by Lemma \ref{non-abelian-hazewinkel-1}. Now, define the arrow
\begin{equation}
\label{non-abelian-hazewinkel-3/2+e}
H_{L/K}^{(\varphi)} : U_{\widetilde{\mathbb X}(L/K)}^\diamond/Y_{L/K}
\rightarrow\text{Gal}(L/K)
\end{equation}
by
\begin{equation}
\label{non-abelian-hazewinkel-3/2+e-def}
H_{L/K}^{(\varphi)} : U.Y_{L/K}\mapsto \sigma_U ,
\end{equation}
for every $U\in U_{\widetilde{\mathbb X}(L/K)}^\diamond$. This arrow is 
clearly a well-defined mapping. In fact, suppose that 
$U,V\in U_{\widetilde{\mathbb X}(L/K)}^\diamond$ satisfy
$U\equiv V\pmod{Y_{L/K}}$. Then $\sigma_U=\sigma_V$. In fact, let 
$Y\in Y_{L/K}$ such that $U=V.Y$. The definition of $Y_{L/K}$ given in 
eq. (\ref{definition-of-Y}) forces that
$Y^{1-\varphi}\in Z_{L/K}(\{K_i,f_i\})$. Thus, the equalities $U^{1-\varphi}
=(V.Y)^{1-\varphi}=V^{1-\varphi}Y^{1-\varphi}$ yield 
$U^{1-\varphi}Z_{L/K}(\{K_i,f_i\})=V^{1-\varphi}Z_{L/K}(\{K_i,f_i\})$, which
shows that 
$\ell_{L/K}^{(\varphi)}(\sigma_U)=\ell_{L/K}^{(\varphi)}(\sigma_V)$ by
eq. (\ref{non-abelian-hazewinkel-3/2}). 
Now, by Lemma \ref{non-abelian-hazewinkel-1}, it follows that 
$\sigma_U=\sigma_V$.
\begin{lemma}
\label{non-abelian-hazewinkel-2}
Suppose that the local field $K$ satisfies the condition given in eq. 
(\ref{rootofunity}). The arrow
\begin{equation*}
H_{L/K}^{(\varphi)} : U_{\widetilde{\mathbb X}(L/K)}^\diamond/Y_{L/K}
\rightarrow\text{Gal}(L/K)
\end{equation*}
defined for the extension $L/K$ is a bijection.
\end{lemma}
\begin{proof}
Choose $U,V\in U_{\widetilde{\mathbb X}(L/K)}^\diamond$ satisfying
$H_{L/K}^{(\varphi)}(U.Y_{L/K})=H_{L/K}^{(\varphi)}(V.Y_{L/K})$. 
Thus, $\sigma_U=\sigma_V$ by the definition eq. 
(\ref{non-abelian-hazewinkel-3/2+e-def}) of the arrow given in eq. 
(\ref{non-abelian-hazewinkel-3/2+e}). 
Now, eq. (\ref{non-abelian-hazewinkel-3/2}) yields
\begin{equation*}
U^{1-\varphi}.Z_{L/K}(\{K_i,f_i\})=V^{1-\varphi}.Z_{L/K}(\{K_i,f_i\}),
\end{equation*}
which proves that $(V^{-1}U)^{1-\varphi}\in Z_{L/K}(\{K_i,f_i\})$.
The equality 
$U.Y_{L/K}=V.Y_{L/K}$ follows immediately by eq. (\ref{definition-of-Y}).
Now, choose any $\sigma\in\text{Gal}(L/K)$. 
By Theorem \ref{fundamental-equation}, there exists 
$U\in U_{\widetilde{\mathbb X}(L/K)}^{\diamond}$ which is unique modulo
$U_{\mathbb X(L/K)}$ (so unique modulo $Y_{L/K}$ as 
$U_{\mathbb X(L/K)}\subseteq Y_{L/K}$), such that
\begin{equation*}
\Pi_{\varphi;L/K}^{\sigma-1}.Z_{L/K}(\{K_i,f_i\})
=U^{1-\varphi}.Z_{L/K}(\{K_i,f_i\}).
\end{equation*}
Thus, by Theorem \ref{non-abelian-hazewinkel-1} and by 
eq. (\ref{non-abelian-hazewinkel-3/2}),
\begin{equation*}
\ell_{L/K}^{(\varphi)}(\sigma)=\ell_{L/K}^{(\varphi)}(\sigma_U),
\end{equation*}
which forces the equality $\sigma=\sigma_U$ for
$U\in U_{\widetilde{\mathbb X}(L/K)}^\diamond$.
\end{proof}
Now, consider the composition of the arrows
\begin{equation}
\Phi_{L/K}^{(\varphi)}:\text{Gal}(L/K)\xrightarrow{\phi_{L/K}^{(\varphi)}} 
U_{\widetilde{\mathbb X}{(L/K)}}^\diamond/U_{\mathbb X(L/K)}
\xrightarrow{c_{L/K}}
U_{\widetilde{\mathbb X}{(L/K)}}^\diamond/Y_{L/K} .
\end{equation}
\begin{lemma}
\label{inverses}
There are the following equalities.
\begin{itemize}
\item[(i)]
$U_{\sigma_U}.Y_{L/K}=U.Y_{L/K}$ for every 
$U\in U_{\widetilde{\mathbb X}(L/K)}^\diamond$;
\item[(ii)]
$\sigma_{U_\sigma}=\sigma$ for every $\sigma\in\text{Gal}(L/K)$.
\end{itemize}
\end{lemma}
\begin{proof}
To prove (i),
let $U\in U_{\widetilde{\mathbb X}(L/K)}^\diamond$. Then, by
eq. (\ref{non-abelian-hazewinkel-3/2}), there 
exists a unique $\sigma_U\in\text{Gal}(L/K)$ satisfying
\begin{equation}
\label{non-abelian-hazewinkel-bijection-1}
U^{1-\varphi}.Z_{L/K}(\{K_i,f_i\})=\ell_{L/K}^{(\varphi)}(\sigma_U)=
\Pi_{\varphi;L/K}^{\sigma_U-1}.Z_{L/K}(\{K_i,f_i\}).
\end{equation}
The equality on the right-hand-side follows from the definition of the
mapping
$\ell_{L/K}^{(\varphi)}:\text{Gal}(L/K)\rightarrow 
U_{\widetilde{\mathbb X}(L/K)}^1/Z_{L/K}(\{K_i,f_i\})$
given in Lemma \ref{non-abelian-hazewinkel-1}. Now, by 
Lemma \ref{fundamental-equation}, for this $\sigma_U\in\text{Gal}(L/K)$,
there exists 
$U_{\sigma_U}\in U_{\widetilde{\mathbb X}(L/K)}^\diamond$, which is unique
modulo $U_{\mathbb X(L/K)}$, satisfying 
\begin{equation*}
U_{\sigma_U}^{1-\varphi}=\Pi_{\varphi;L/K}^{\sigma_U-1} .
\end{equation*}
Thus,
\begin{equation*}
U_{\sigma_U}^{1-\varphi}.Z_{L/K}(\{K_i,f_i\})=
U^{1-\varphi}.Z_{L/K}(\{K_i,f_i\}),
\end{equation*}
by eq. (\ref{non-abelian-hazewinkel-bijection-1}), which proves that
\begin{equation}
\label{non-abelian-hazewinkel-bijection-i}
U_{\sigma_U}.Y_{L/K}=U.Y_{L/K},
\end{equation}
by the definition of $Y_{L/K}$ given in eq. (\ref{definition-of-Y}).
Moreover, as $U_{\mathbb X(L/K)}\subseteq Y_{L/K}$, this equality
(\ref{non-abelian-hazewinkel-bijection-i})
does not depend on the choice of $U_{\sigma_U}$ modulo $U_{\mathbb X(L/K)}$.
Now, for (ii),
let $\sigma\in\text{Gal}(L/K)$. By Lemma \ref{fundamental-equation}, there
exists $U_\sigma\in U_{\widetilde{\mathbb X}(L/K)}^\diamond$, which is unique
modulo $U_{\mathbb X(L/K)}$, such that
\begin{equation}
\label{non-abelian-hazewinkel-bijection-2}
U_{\sigma}^{1-\varphi}=\Pi_{\varphi;L/K}^{\sigma-1} .
\end{equation}
For any such $U_\sigma\in U_{\widetilde{\mathbb X}(L/K)}^\diamond$, there 
exists a unique $\sigma_{U_\sigma}\in\text{Gal}(L/K)$ satisfying
\begin{equation*}
U_\sigma^{1-\varphi}.Z_{L/K}(\{K_i,f_i\})
=\ell_{L/K}^{(\varphi)}(\sigma_{U_\sigma})
\end{equation*}
by eq. (\ref{non-abelian-hazewinkel-3/2}).
Thus, by eq. (\ref{non-abelian-hazewinkel-bijection-2}) and Lemma
\ref{non-abelian-hazewinkel-1}, it follows that
\begin{equation*}
\ell_{L/K}^{(\varphi)}(\sigma_{U_\sigma})=\Pi_{\varphi;L/K}^{\sigma-1}.
Z_{L/K}(\{K_i,f_i\})=\ell_{L/K}^{(\varphi)}(\sigma),
\end{equation*}
which proves that $\sigma_{U_\sigma}=\sigma$.
\end{proof}
Lemma \ref{inverses} immediately yields
\begin{equation*}
H_{L/K}^{(\varphi)}\circ\Phi_{L/K}^{(\varphi)}=\text{id}_{\text{Gal}(L/K)};
\end{equation*}
and
\begin{equation*}
\Phi_{L/K}^{(\varphi)}\circ H_{L/K}^{(\varphi)}=
\text{id}_{U_{\widetilde{\mathbb X}(L/K)}^\diamond/Y_{L/K}}.
\end{equation*}
The following theorem follows from Lemma \ref{non-abelian-hazewinkel-2},
Lemma \ref{inverses}, Theorem \ref{cocycle1} and from the fact that 
$U_{\mathbb X(L/K)}$ is a topological $\text{Gal}(L/K)$-submodule 
of $Y_{L/K}$.
\begin{theorem}[Fesenko]
\label{cocycle2}
Suppose that the local field $K$ satisfies the condition given in eq. 
(\ref{rootofunity}). The mapping
\begin{equation*}
\Phi_{L/K}^{(\varphi)}:\text{Gal}(L/K)\rightarrow 
U_{\widetilde{\mathbb X}{(L/K)}}^\diamond/Y_{L/K}
\end{equation*}
defined for the extension $L/K$
is a bijection with the inverse 
\begin{equation*}
H_{L/K}^{(\varphi)}:U_{\widetilde{\mathbb X}{(L/K)}}^\diamond/Y_{L/K}
\rightarrow\text{Gal}(L/K).
\end{equation*} 
For every $\sigma,\tau\in\text{Gal}(L/K)$,
\begin{equation}
\Phi_{L/K}^{(\varphi)}(\sigma\tau)=\Phi_{L/K}^{(\varphi)}(\sigma)
\Phi_{L/K}^{(\varphi)}(\tau)^\sigma
\end{equation}
co-cycle condition is satisfied. 
\end{theorem}
By Corollary \ref{grouplaw1}, Theorem \ref{cocycle2} has the following
consequence.
\begin{corollary}
\label{grouplaw2}
Define a law of composition $\ast$ on 
$U_{\widetilde{\mathbb X}{(L/K)}}^\diamond/Y_{L/K}$ by
\begin{equation}
\label{star2}
\overline{U}\ast\overline{V}=\overline{U}.
\overline{V}^{(\Phi_{L/K}^{(\varphi)})^{-1}(\overline{U})}
\end{equation}
for every $\overline{U}=U.Y_{L/K}, \overline{V}=V.Y_{L/K}\in 
U_{\widetilde{\mathbb X}{(L/K)}}^\diamond/Y_{L/K}$ with $U,V\in
U_{\widetilde{\mathbb X}{(L/K)}}^\diamond$. Then
$U_{\widetilde{\mathbb X}{(L/K)}}^\diamond/Y_{L/K}$ is a topological 
group under $\ast$, and the map $\Phi_{L/K}^{(\varphi)}$ induces 
an isomorphism of topological groups
\begin{equation}
\Phi_{L/K}^{(\varphi)}:\text{Gal}(L/K)\xrightarrow{\sim}
U_{\widetilde{\mathbb X}{(L/K)}}^\diamond/Y_{L/K}, 
\end{equation}
where the topological group structure on 
$U_{\widetilde{\mathbb X}{(L/K)}}^\diamond/Y_{L/K}$ is defined with
respect to the binary operation $\ast$ defined by eq. (\ref{star2}).
\end{corollary}
\begin{definition}
Let $K$ be a local field satisfying the condition given in eq. 
(\ref{rootofunity}).
Let $L/K$ be a totally-ramified $APF$-Galois extension satisfying
eq. (\ref{phitower}). 
The mapping 
\begin{equation*}
\Phi_{L/K}^{(\varphi)}:\text{Gal}(L/K)\rightarrow
U_{\widetilde{\mathbb X}{(L/K)}}^\diamond/Y_{L/K},
\end{equation*} 
defined in Theorem \ref{cocycle2}, is called the 
\textit{Fesenko reciprocity map for the extension $L/K$}.
\end{definition}
For each $0\leq i\in\mathbb R$, we have previously introduced the groups 
$\left(U_{\widetilde{\mathbb X}{(L/K)}}^\diamond\right)^i$. 
For $0\leq n\in\mathbb Z$, let 
\begin{equation}
Q_{L/K}^n=c_{L/K}\left(\left(U_{\widetilde{\mathbb X}{(L/K)}}
^\diamond\right)^nU_{\mathbb X(L/K)}/U_{\mathbb X(L/K)}\cap
\text{im}(\phi_{L/K}^{(\varphi)})\right),
\end{equation}
which is a subgroup of 
$\left(U_{\widetilde{\mathbb X}{(L/K)}}^\diamond\right)^nY_{L/K}/Y_{L/K}$.
Now, Fesenko ramification theorem, stated in Theorem \ref{ramification1}, 
can be reformulated for the reciprocity map $\Phi_{L/K}^{(\varphi)}$ 
corresponding to the extension $L/K$ as follows.
\begin{theorem}[Ramification theorem]
\label{ramification2}
Suppose that the local field $K$ satisfies the condition given in
eq. (\ref{rootofunity}). 
For $0\leq n\in\mathbb Z$,
let $\text{Gal}(L/K)_n$ denote the $n^{th}$ higher ramification 
subgroup of the Galois group $\text{Gal}(L/K)$ corresponding to the
$APF$-Galois sub-extension $L/K$ of $K_\varphi/K$ in the lower 
numbering. Then, there exists the inclusion
\begin{equation*}
\Phi_{L/K}^{(\varphi)}
\left(\text{Gal}(L/K)_n-\text{Gal}(L/K)_{n+1}\right)\subseteq
\left(U_{\widetilde{\mathbb X}{(L/K)}}^\diamond\right)^nY_{L/K}/Y_{L/K} -
Q_{L/K}^{n+1} .
\end{equation*}
\end{theorem}
\begin{proof}
Let $\tau\in\text{Gal}(L/K)_n$. The first half of the proof of
Proposition 1 in \cite{fesenko-2001} shows that 
$\Phi_{L/K}^{(\varphi)}(\tau)\in\left(U_{\widetilde{\mathbb X}{(L/K)}}
^\diamond\right)^nY_{L/K}/Y_{L/K}$.
Now, let $\overline{U}=U.Y_{L/K}\in Q_{L/K}^{n+1}$, where 
$U\in U_{\widetilde{\mathbb X}{(L/K)}}^\diamond$. Then, by the definition of
$Q_{L/K}^{n+1}$, there exists $V\in\left(U_{\widetilde{\mathbb X}{(L/K)}}
^\diamond\right)^{n+1}U_{\mathbb X(L/K)}$ and $\tau\in\text{Gal}(L/K)$ 
such that
$c_{L/K}(\overline{V})=\overline{U}$ and 
$\phi_{L/K}^{(\varphi)}(\tau)=\overline{V}$, where 
$\overline{V}=V.U_{\mathbb X(L/K)}$.
So, $\Phi_{L/K}^{(\varphi)}(\tau)=\overline{U}$.
The second half of the proof of Proposition 1 in \cite{fesenko-2001} now 
proves that $\tau\in\text{Gal}(L/K)_{n+1}$.
\end{proof}
Now, let $M/K$ be a Galois sub-extension of $L/K$. Thus, there exists
the chain of field extensions
\begin{equation*}
K\subseteq M\subseteq L\subseteq K_\varphi ,
\end{equation*}
where $M$ is a totally-ramified $APF$-Galois extension over the local
field $K$ satisfying the condition given in eq. (\ref{rootofunity}) by
Lemma \ref{apftower}.
 
Now, the basic ascending chain of sub-extensions in $L/K$ fixed 
in eq. (\ref{basic-ascending-chain-L/K}) restricted to $M$ 
\begin{equation}
\label{basic-ascending-chain-M/K}
K=K_o\cap M\subseteq K_1\cap M\subseteq\cdots\subseteq K_i\cap M
\subseteq\cdots\subseteq L\cap M=M
\end{equation}
is \textit{almost} a basic ascending chain of sub-extensions in $M/K$
(almost in the sense that, there may exist elements $0\leq i\in\mathbb Z$ 
such that $K_i\cap M=K_{i+1}\cap M$).
In fact, for each $0\leq i\in\mathbb Z$,
the extension $K_i\cap M/K$ is clearly Galois. 
For each $0\leq i\in\mathbb Z$, consider the surjective homomorphism
\begin{equation*}
r_{K_{i+1}\cap M}:\text{Gal}(K_{i+1}/K_i)\twoheadrightarrow
\text{Gal}(K_{i+1}\cap M/K_i\cap M)
\end{equation*}
defined by the restriction to $K_{i+1}\cap M$ as
\begin{equation*}
\sigma\mapsto\sigma\mid_{K_{i+1}\cap M}
\end{equation*} 
for every $\sigma\in\text{Gal}(K_{i+1}/K_i)$. As 
$\text{Gal}(K_{i+1}/K_i)$ is cyclic of prime order 
$p=\text{char}(\kappa_K)$ (resp. of order relatively prime to $p$) 
in case $1\leq i\in\mathbb Z$ (resp. in case $i=0$), it follows that
$\text{Gal}(K_{i+1}\cap M/K_i\cap M)$ is cyclic of order $p$ or $1$
(resp. of order relatively prime to $p$) in case $1\leq i\in\mathbb Z$
(resp. in case $i=0$). Now fix this almost basic ascending chain of 
sub-extensions in $M/K$ introduced in eq. (\ref{basic-ascending-chain-M/K}).
Observe that, for each $1\leq i\in\mathbb Z$, 
$\sigma_i\mid_{\widetilde{M}}\in\text{Gal}(\widetilde{M}/\widetilde{K})$
satisfies
\begin{equation*}
<(\sigma_i\mid_{\widetilde M})\mid_{K_i\cap M}=\sigma_i\mid_{K_i\cap M}>
=\text{Gal}(K_i\cap M/ K_{i-1}\cap M)
\end{equation*}
as the restriction map 
$r_{K_i\cap M}:\text{Gal}(K_i/K_{i-1})\twoheadrightarrow
\text{Gal}(K_i\cap M/K_{i-1}\cap M)$ is a surjective homomorphism and
$<\sigma_i\mid_{K_i}>=\text{Gal}(K_i/K_{i-1})$.
As usual, we set $\widetilde{K_i\cap M}=(K_i\cap M)\widetilde{K}$.
Note that, for each $1\leq k\in\mathbb Z$, the norm map
\begin{equation*}
\widetilde{N}_{K_{k+1}/K_{k+1}\cap M}:U_{\widetilde{K}_{k+1}}\rightarrow 
U_{\widetilde{K_{k+1}\cap M}}
\end{equation*} 
induces a homomorphism
\begin{equation*}
\widetilde{N}_{K_{k+1}/K_{k+1}\cap M}^*:
U_{\widetilde{K}_{k+1}}/U_{\widetilde{K}_{k+1}}^{\sigma_{k+1}-1}
\rightarrow 
U_{\widetilde{K_{k+1}\cap M}}/U_{\widetilde{K_{k+1}\cap M}}
^{\sigma_{k+1}\mid_{\widetilde{M}}-1},
\end{equation*}
defined by
\begin{equation*}
\widetilde{N}_{K_{k+1}/K_{k+1}\cap M}^*:
u.U_{\widetilde{K}_{k+1}}^{\sigma_{k+1}-1}\mapsto
\widetilde{N}_{K_{k+1}/K_{k+1}\cap M}(u).U_{\widetilde{K_{k+1}\cap M}}
^{\sigma_{k+1}\mid_{\widetilde{M}}-1},
\end{equation*}
for every $u\in U_{\widetilde{K}_{k+1}}$, as 
$\widetilde{N}_{K_{k+1}/K_{k+1}\cap M}
\left(U_{\widetilde{K}_{k+1}}^{\sigma_{k+1}-1}\right)\subseteq
U_{\widetilde{K_{k+1}\cap M}}^{\sigma_{k+1}\mid_{\widetilde{M}}-1}$. 
Thus, the following square,
where the upper and lower horizontal arrows are defined by eq.s 
(\ref{injective-map}) and (\ref{injective-map-definition}),
\begin{equation*}
\SelectTips{cm}{}\xymatrix{
{\text{Gal}(K_{k+1}/K_k)}\ar[r]^{\Xi_{K_{k+1}/K_k}}
\ar[d]_{r_{K_{k+1}\cap M}} & 
{U_{\widetilde{K}_{k+1}}/U_{\widetilde{K}_{k+1}}^{\sigma_{k+1}-1}}
\ar[d]^{\widetilde{N}_{K_{k+1}/K_{k+1}\cap M}^*} \\
{\text{Gal}(K_{k+1}\cap M/K_k\cap M)}\ar[r]^{\Xi_{K_{k+1}\cap M/K_k\cap M}} & 
{U_{\widetilde{K_{k+1}\cap M}}/U_{\widetilde{K_{k+1}\cap M}}
^{\sigma_{k+1}\mid_{\widetilde{M}}-1}},
}
\end{equation*} 
is commutative, as
$\widetilde{N}_{K_{k+1}/K_{k+1}\cap M}(\pi_{K_{k+1}})=\pi_{K_{k+1}\cap M}$
by the \textit{norm coherence} of the Lubin-Tate labelling
$(\pi_{K'})_{\underset{[K':K]<\infty}{\underbrace{K\subseteq K'}}
\subset K_\varphi}$. Hence,
\begin{equation*}
\widetilde{N}_{K_{k+1}/K_{k+1}\cap M}^*\left(T_k^{(L/K)}\right)=T_k^{(M/K)}.
\end{equation*}
Now, we shall define an arrow
\begin{equation}
\label{h-M/K}
h_k^{(M/K)}:\prod_{1\leq i\leq k}
U_{\widetilde{K_k\cap M}}^{\sigma_i\mid_{\widetilde{M}}-1}\rightarrow
\left(\prod_{1\leq i\leq k+1}U_{\widetilde{K_{k+1}\cap M}}
^{\sigma_i\mid_{\widetilde{M}}-1}\right)/U_{\widetilde{K_{k+1}\cap M}}
^{\sigma_{k+1}\mid_{\widetilde{M}}-1}
\end{equation}
which splits the exact sequence
\begin{equation}
\label{split-short-exact-M/K}
\SelectTips{cm}{}\xymatrix@1{
1\ar[r] & T_k^{(M/K)'}\ar[r] & 
{\left(\prod_{1\leq i\leq k+1}U_{\widetilde{K_{k+1}\cap M}}
^{\sigma_i\mid_{\widetilde{M}}-1}\right)
/U_{\widetilde{K_{k+1}\cap M}}^{\sigma_{k+1}\mid_{\widetilde{M}}-1}}
\ar[r]^-{\widetilde{N}_{K_{k+1}\cap M/K_k\cap M}} &
{\prod_{1\leq i\leq k}U_{\widetilde{K_k\cap M}}
^{\sigma_i\mid_{\widetilde{M}}-1}}
\ar[r]\ar@/_2.2pc/[l]_{h_k^{(M/K)}}&1}
\end{equation}
in such a way that
\begin{equation}
\label{diagram-h}
\SelectTips{cm}{}\xymatrix{
{\left(\prod_{1\leq i\leq k+1}U_{\widetilde{K}_{k+1}}
^{\sigma_i-1}\right)/U_{\widetilde{K}_{k+1}}^{\sigma_{k+1}-1}}
\ar[r]^-{\widetilde{N}_{K_{k+1}/K_k}}
\ar[d]_{\widetilde{N}^*_{K_{k+1}/K_{k+1}\cap M}} & 
{\prod_{1\leq i\leq k}U_{\widetilde{K}_k}^{\sigma_i-1}}
\ar@/_2.2pc/[l]_{h_k^{(L/K)}}
\ar[d]^{\widetilde{N}_{K_k/K_k\cap M}} \\
{\left(\prod_{1\leq i\leq k+1}U_{\widetilde{K_{k+1}\cap M}}
^{\sigma_i\mid_{\widetilde{M}}-1}\right)
/U_{\widetilde{K_{k+1}\cap M}}^{\sigma_{k+1}\mid_{\widetilde{M}}-1}}
\ar[r]^-{\widetilde{N}_{K_{k+1}\cap M/K_k\cap M}} & 
{\prod_{1\leq i\leq k}U_{\widetilde{K_k\cap M}}
^{\sigma_i\mid_{\widetilde{M}}-1}}
\ar@/^2.2pc/[l]^{h_k^{(M/K)}}
}
\end{equation}
is a commutative square. In order to do so, however, closely following 
Fesenko (\cite{fesenko-2000, fesenko-2001, fesenko-2005}),
let us review the construction of a splitting 
\begin{equation*}
h_k^{(L/K)} :\prod_{1\leq i\leq k}U_{\widetilde{K}_k}^{\sigma_i-1}\rightarrow
\left(\prod_{1\leq i\leq k+1}U_{\widetilde{K}_{k+1}}^{\sigma_i-1}\right)/
U_{\widetilde{K}_{k+1}}^{\sigma_{k+1}-1} 
\end{equation*}
of the short exact sequence
\begin{equation}
\label{split-short-exact-L/K}
\SelectTips{cm}{}\xymatrix@1{
1\ar[r] & T_k^{(L/K)'}\ar[r] & 
{\left(\prod_{1\leq i\leq k+1}U_{\widetilde{K}_{k+1}}
^{\sigma_i-1}\right)/U_{\widetilde{K}_{k+1}}^{\sigma_{k+1}-1}}
\ar[r]^-{\widetilde{N}_{K_{k+1}/K_k}} &
{\prod_{1\leq i\leq k}U_{\widetilde{K}_k}^{\sigma_i-1}}
\ar[r]\ar@/_2.2pc/[l]_{h_k^{(L/K)}}&1.}
\end{equation}
The product module $\prod_{1\leq i\leq k}U_{\widetilde{K}_k}^{\sigma_i-1}$
is a closed $\mathbb Z_p$-submodule of $U_{\widetilde{K}_k}^1$. Let 
$\{\lambda_j\}$ be a system of topological multiplicative generators of
the topological $\mathbb Z_p$-module $\prod_{1\leq i\leq k}U_{\widetilde{K}_k}
^{\sigma_i-1}$ satisfying the following property. If the torsion 
$\left(\prod_{1\leq i\leq k}U_{\widetilde{K}_k}^{\sigma_i-1}
\right)_{\text{tor}}$ of the module
$\prod_{1\leq i\leq k}U_{\widetilde{K}_k}^{\sigma_i-1}$ is non-trivial, there
exists $\lambda_*\in\{\lambda_j\}$ of order $p^m$ in the torsion of the 
module while the remaining $\lambda_j$ ($j\neq *$) are topologically 
independent over $\mathbb Z_p$.
Now, define a map
\begin{equation*}
h_k^{(L/K)} :\{\lambda_j\}\rightarrow \left(\prod_{1\leq i\leq k+1}
U_{\widetilde{K}_{k+1}}^{\sigma_i-1}\right)
/U_{\widetilde{K}_{k+1}}^{\sigma_{k+1}-1}
\end{equation*} 
on the topological generators $\{\lambda_j\}$ by
\begin{equation*}
h_k^{(L/K)} :\lambda_j\mapsto u_j.U_{\widetilde{K}_{k+1}}^{\sigma_{k+1}-1},
\end{equation*}
where $u_j\in\prod_{1\leq i\leq k+1}U_{\widetilde{K}_{k+1}}
^{\sigma_i-1}$ satisfies $\widetilde{N}_{K_{k+1}/K_k}(u_j)=\lambda_j$.
It then follows by step 5 of the proof of the Theorem in Section 3
of \cite{fesenko-2005} that,
$h_k^{(L/K)}(\lambda_*)^{p^m}\in U_{\widetilde{K}_{k+1}}^{\sigma_{k+1}-1}$.
Therefore, the arrow $h_k^{(L/K)}$ extends uniquely to a homomorphism
\begin{equation*}
h_k^{(L/K)} :\prod_{1\leq i\leq k}U_{\widetilde{K}_k}^{\sigma_i-1}\rightarrow
\left(\prod_{1\leq i\leq k+1}U_{\widetilde{K}_{k+1}}^{\sigma_i-1}\right)/
U_{\widetilde{K}_{k+1}}^{\sigma_{k+1}-1} ,
\end{equation*}
which is a splitting of the short exact sequence given by
eq. (\ref{split-short-exact-L/K}).
Now, define 
\begin{equation*}
h_k^{(M/K)}:\prod_{1\leq i\leq k}
U_{\widetilde{K_k\cap M}}^{\sigma_i\mid_{\widetilde{M}}-1}\rightarrow
\left(\prod_{1\leq i\leq k+1}U_{\widetilde{K_{k+1}\cap M}}
^{\sigma_i\mid_{\widetilde{M}}-1}\right)/U_{\widetilde{K_{k+1}\cap M}}
^{\sigma_{k+1}\mid_{\widetilde{M}}-1}
\end{equation*}
as follows. Note that,
\begin{equation*}
\widetilde{N}_{K_k/K_k\cap M}:\prod_{1\leq i\leq k}U_{\widetilde{K}_k}
^{\sigma_i-1}\rightarrow\prod_{1\leq i\leq k}U_{\widetilde{K_k\cap M}}
^{\sigma_i\mid_{\widetilde{M}}-1}
\end{equation*}
is a surjective homomorphism, as 
$\widetilde{N}_{K_k/K_k\cap M}:U_{\widetilde{K}_k}\rightarrow 
U_{\widetilde{K_k\cap M}}$ is a surjective homomorphism.
Thus, the collection $\{\widetilde{N}_{K_k/K_k\cap M}(\lambda_j)\}$
is a system of topological multiplicative generators of the topological
$\mathbb Z_p$-module $\prod_{1\leq i\leq k}U_{\widetilde{K_k\cap M}}
^{\sigma_i\mid_{\widetilde{M}}-1}$. Moreover, note that
$\widetilde{N}_{K_k/K_k\cap M}(\lambda_*)^{p^m}=1$. Thus,
$\widetilde{N}_{K_k/K_k\cap M}(\lambda_*)$ is in the torsion-part
$\left(\prod_{1\leq i\leq k}U_{\widetilde{K_k\cap M}}
^{\sigma_i\mid_{\widetilde{M}}-1}\right)_{\text{tor}}$.
For the remaining $\lambda_j$ ($j\neq *$), the collection
$\left\{\widetilde{N}_{K_k/K_k\cap M}(\lambda_j)\right\}_{j\neq *}$ 
is topologically independent over $\mathbb Z_p$. Now, following Fesenko's 
construction of $h_k^{(L/K)}$, define a map
\begin{equation*}
h_k^{(M/K)} :\left\{\widetilde{N}_{K_k/K_k\cap M}(\lambda_j)\right\}
\rightarrow \left(\prod_{1\leq i\leq k+1}
U_{\widetilde{K_{k+1}\cap M}}^{\sigma_i\mid_{\widetilde{M}}-1}\right)
/U_{\widetilde{K_{k+1}\cap M}}^{\sigma_{k+1}\mid_{\widetilde{M}}-1}
\end{equation*} 
on the topological generators 
$\left\{\widetilde{N}_{K_k/K_k\cap M}(\lambda_j)\right\}$
by
\begin{equation}
\label{def-split-M/K}
h_k^{(M/K)} : \widetilde{N}_{K_k/K_k\cap M}(\lambda_j)\mapsto
\widetilde{N}_{K_{k+1}/K_{k+1}\cap M}(u_j).U_{\widetilde{K_{k+1}\cap M}}
^{\sigma_{k+1}\mid_{\widetilde{M}}-1},
\end{equation}
where $u_j\in\prod_{1\leq i\leq k+1}U_{\widetilde{K}_{k+1}}
^{\sigma_i-1}$ satisfies $\widetilde{N}_{K_{k+1}/K_k}(u_j)=\lambda_j$, and
thereby
$\widetilde{N}_{K_{k+1}/K_{k+1}\cap M}(u_j)$ satisfies
\begin{equation*}
\aligned
\widetilde{N}_{K_{k+1}\cap M/K_k\cap M}
\left(\widetilde{N}_{K_{k+1}/K_{k+1}\cap M}(u_j)\right)&=
\widetilde{N}_{K_{k+1}/K_k\cap M}(u_j)\\
&=\widetilde{N}_{K_k/K_k\cap M}\left(\widetilde{N}_{K_{k+1}/K_k}(u_j)\right)\\
&=\widetilde{N}_{K_k/K_k\cap M}(\lambda_j) .
\endaligned
\end{equation*}
Therefore, the arrow $h_k^{(M/K)}$ extends uniquely to a homomorphism
\begin{equation*}
h_k^{(M/K)}:\prod_{1\leq i\leq k}
U_{\widetilde{K_k\cap M}}^{\sigma_i\mid_{\widetilde{M}}-1}\rightarrow
\left(\prod_{1\leq i\leq k+1}U_{\widetilde{K_{k+1}\cap M}}
^{\sigma_i\mid_{\widetilde{M}}-1}\right)/U_{\widetilde{K_{k+1}\cap M}}
^{\sigma_{k+1}\mid_{\widetilde{M}}-1} ,
\end{equation*}
which is a splitting of the short exact sequence given by eq. 
(\ref{split-short-exact-M/K}). In fact, it suffices to show that, for
$u_j\in\prod_{1\leq i\leq k+1}U_{\widetilde{K}_{k+1}}^{\sigma_i-1}$ 
satisfying $\widetilde{N}_{K_{k+1}/K_k}(u_j)=\lambda_j$,
\begin{equation*}
h_k^{(M/K)}\circ \widetilde{N}_{K_{k+1}\cap M/K_k\cap M}:
\widetilde{N}_{K_{k+1}/K_{k+1}\cap M}(u_j)U_{\widetilde{K_{k+1}\cap M}}
^{\sigma_{k+1}\mid_{\widetilde{M}}-1}\mapsto 
\widetilde{N}_{K_{k+1}/K_{k+1}\cap M}(u_j)U_{\widetilde{K_{k+1}\cap M}}
^{\sigma_{k+1}\mid_{\widetilde{M}}-1},
\end{equation*}
which follows from the equalities
\begin{equation*}
\begin{aligned}
h_k^{(M/K)}\left(\widetilde{N}_{K_{k+1}\cap M/K_k\cap M}
(\widetilde{N}_{K_{k+1}/K_{k+1}\cap M}(u_j))\right)&=
h_k^{(M/K)}(\widetilde{N}_{K_{k+1}/K_k\cap M}(u_j))\\
&=h_k^{(M/K)}(\widetilde{N}_{K_k/K_k\cap M}
(\widetilde{N}_{K_{k+1}/K_k}(u_j)))\\
&=h_k^{(M/K)}(\widetilde{N}_{K_k/K_k\cap M}(\lambda_j))
\end{aligned}
\end{equation*}
and by the definition of the arrow $h_k^{(M/K)}$ given by eq. 
(\ref{def-split-M/K}).
Moreover, the diagram (\ref{diagram-h}) commutes, which follows from
the equality
\begin{equation*}
\begin{aligned}
h_k^{(M/K)}(\widetilde{N}_{K_k/K_k\cap M}(\lambda_j))&=
\widetilde{N}_{K_{k+1}/K_{k+1}\cap M}(u_j)U_{\widetilde{K_{k+1}\cap M}}
^{\sigma_{k+1}\mid_{\widetilde{M}}-1}\\
&=\widetilde{N}^*_{K_{k+1}/K_{k+1}\cap M}
(u_j.U_{\widetilde{K}_{k+1}}^{\sigma_{k+1}-1})\\
&=\widetilde{N}^*_{K_{k+1}/K_{k+1}\cap M}(h_k^{(L/K)}(\lambda_j)).
\end{aligned}
\end{equation*}
For each $1\leq k\in\mathbb Z$, consider any map
\begin{equation}
\label{g-M/K}
g_k^{(M/K)}:\prod_{1\leq i\leq k}U_{\widetilde{K_k\cap M}}
^{\sigma_i\mid_{\widetilde{M}}-1}\rightarrow\prod_{1\leq i\leq k+1}
U_{\widetilde{K_{k+1}\cap M}}^{\sigma_i\mid_{\widetilde{M}}-1}
\end{equation}
which commutes the following square
\begin{equation}
\label{diagram-g}
\SelectTips{cm}{}\xymatrix{
{\prod_{1\leq i\leq k}U_{\widetilde{K}_k}^{\sigma_i-1}}
\ar[r]^-{g_k^{(L/K)}}\ar[d]_{\widetilde{N}_{K_k/K_k\cap M}} & 
{\prod_{1\leq i\leq k+1}U_{\widetilde{K}_{k+1}}^{\sigma_i-1}}
\ar[d]^{\widetilde{N}_{K_{k+1}/K_{k+1}\cap M}} \\
{\prod_{1\leq i\leq k}U_{\widetilde{K_k\cap M}}
^{\sigma_i\mid_{\widetilde{M}}-1}}\ar@{.>}[r]^-{g_k^{(M/K)}} & 
{\prod_{1\leq i\leq k+1}U_{\widetilde{K_{k+1}\cap M}}
^{\sigma_i\mid_{\widetilde{M}}-1}}.
}
\end{equation}
Note that, such a map satisfies
\begin{equation*}
h_k^{(M/K)}=g_k^{(M/K)}\mod{U_{\widetilde{K_{k+1}\cap M}}
^{\sigma_{k+1}\mid_{\widetilde{M}}-1}}.
\end{equation*}
In fact, by the commutative diagram (\ref{diagram-h}),
for any $w\in\prod_{1\leq i\leq k}U_{\widetilde{K_k\cap M}}
^{\sigma_i\mid_{\widetilde{M}}-1}$, there exists 
$v\in\prod_{1\leq i\leq k}U_{\widetilde{K}_k}^{\sigma_i-1}$ such that
$w=\widetilde{N}_{K_k/K_k\cap M}(v)$, and
\begin{equation*}
\begin{aligned}
h_k^{(M/K)}(w)&=h_k^{(M/K)}\left(\widetilde{N}_{K_k/K_k\cap M}(v)\right)\\
&=\widetilde{N}^*_{K_{k+1}/K_{k+1}\cap M}\left(h_k^{(L/K)}(v)\right)\\
&=\widetilde{N}^*_{K_{k+1}/K_{k+1}\cap M}\left(g_k^{(L/K)}(v).
U_{\widetilde{K}_{k+1}}^{\sigma_{k+1}-1}\right)\\
&=\widetilde{N}_{K_{k+1}/K_{k+1}\cap M}\left(g_k^{(L/K)}(v)\right).
U_{\widetilde{K_{k+1}\cap M}}^{\sigma_{k+1}\mid_{\widetilde{M}}-1},
\end{aligned}
\end{equation*}
and by the commutativity of the diagram (\ref{diagram-g}),
\begin{equation*}
\begin{aligned}
\widetilde{N}_{K_{k+1}/K_{k+1}\cap M}\left(g_k^{(L/K)}(v)\right)&=
g_k^{(M/K)}\left(\widetilde{N}_{K_k/K_k\cap M}(v)\right)\\
&=g_k^{(M/K)}(w).
\end{aligned}
\end{equation*}
Thus, the equality
\begin{equation*}
h_k^{(M/K)}(w)=g_k^{(M/K)}(w).U_{\widetilde{K_{k+1}\cap M}}
^{\sigma_{k+1}\mid_{\widetilde{M}}-1}
\end{equation*}
follows for every 
$w\in\prod_{1\leq i\leq k}U_{\widetilde{K_k\cap M}}
^{\sigma_i\mid_{\widetilde{M}}-1}$.

Now, for each $1\leq i\in\mathbb Z$, introduce the map
\begin{equation*}
f_i^{(M/K)}: U_{\widetilde{K_i\cap M}}^{\sigma_i\mid_{\widetilde{M}}-1}
\rightarrow U_{\widetilde{\mathbb X}(M/K)}
\end{equation*}
by
\begin{equation*}
f_i^{(M/K)}(w)=\widetilde{\mathcal N}_{L/M}\left(f_i^{(L/K)}(v)\right),
\end{equation*}
where $v\in U_{\widetilde{K}_i}^{\sigma_i-1}$ is any element satisfying
$\widetilde{N}_{K_i/K_i\cap M}(v)=w\in U_{\widetilde{K_i\cap M}}
^{\sigma_i\mid_{\widetilde{M}}-1}$. Note that, if 
$v'\in U_{\widetilde{K}_i}^{\sigma_i-1}$ such that
$\widetilde{N}_{K_i/K_i\cap M}(v')=w$, then
$\widetilde{\mathcal N}_{L/M}\left(f_i^{(L/K)}(v)\right)=
\widetilde{\mathcal N}_{L/M}\left(f_i^{(L/K)}(v')\right)$. In fact, there 
exists $u\in\ker\left(\widetilde{N}_{K_i/K_i\cap M}\right)$ such that
$v'=vu$. Thus, we have to verify that
$\widetilde{\mathcal N}_{L/M}\left(f_i^{(L/K)}(v)\right)=
\widetilde{\mathcal N}_{L/M}\left(f_i^{(L/K)}(vu)\right)$. That is,
for each $1\leq j\in\mathbb Z$, we have to check that
$\widetilde{N}_{K_j/K_j\cap M}\left(\text{Pr}_{\widetilde{K}_j}
(f_i^{(L/K)}(v))\right)
=\widetilde{N}_{K_j/K_j\cap M}\left(\text{Pr}_{\widetilde{K}_j}
(f_i^{(L/K)}(vu))\right)$.
Now, for $j>i$, it follows that
\begin{equation*}
\begin{aligned}
\widetilde{N}_{K_j/K_j\cap M}\left(\text{Pr}_{\widetilde{K}_j}
(f_i^{(L/K)}(v))\right)&=
\widetilde{N}_{K_j/K_j\cap M}\left(g_{j-1}^{(L/K)}\circ\cdots
\circ g_i^{(L/K)}(v)\right)\\
&=g_{j-1}^{(M/K)}\circ\cdots\circ g_i^{(M/K)}
\left(\widetilde{N}_{K_i/K_i\cap M}(v)\right)\\
&= g_{j-1}^{(M/K)}\circ\cdots\circ g_i^{(M/K)}
\left(\widetilde{N}_{K_i/K_i\cap M}(vu)\right)\\
&=\widetilde{N}_{K_j/K_j\cap M}\left(g_{j-1}^{(L/K)}\circ\cdots
\circ g_i^{(L/K)}(vu)\right)\\
&=\widetilde{N}_{K_j/K_j\cap M}\left(\text{Pr}_{\widetilde{K}_j}
(f_i^{(L/K)}(vu))\right) .
\end{aligned}
\end{equation*}
Thus, the map 
\begin{equation*}
f_i^{(M/K)}: U_{\widetilde{K_i\cap M}}^{\sigma_i\mid_{\widetilde{M}}-1}
\rightarrow U_{\widetilde{\mathbb X}(M/K)}
\end{equation*}
is well-defined. Moreover, for $j>i$,
\begin{equation*}
\text{Pr}_{\widetilde{K_j\cap M}}\circ f_i^{(M/K)}=\left(g_{j-1}^{(M/K)}\circ
\cdots\circ g_i^{(M/K)}\right)\mid_{U_{\widetilde{K_i\cap M}}
^{\sigma_i\mid_{\widetilde{M}}-1}}.
\end{equation*}
In fact, for $w\in U_{\widetilde{K_i\cap M}}^{\sigma_i\mid_{\widetilde{M}}-1}$,
there exists $v\in U_{\widetilde{K}_i}^{\sigma_i-1}$ such that
$\widetilde{N}_{K_i/K_i\cap M}(v)=w$, and 
$f_i^{(M/K)}(w)=\widetilde{\mathcal N}_{L/M}\left(f_i^{(L/K)}(v)\right)$. 
That is, the following square
\begin{equation}
\label{square-L/M}
\SelectTips{cm}{}\xymatrix{
{U_{\widetilde{K}_i}^{\sigma_i-1}}\ar[r]^{f_i^{(L/K)}}
\ar[d]_{\widetilde{N}_{K_i/K_i\cap M}} & 
{U_{\widetilde{\mathbb X}(L/K)}}\ar[d]^{\widetilde{\mathcal N}_{L/M}} \\
{U_{\widetilde{K_i\cap M}}^{\sigma_i\mid_{\widetilde{M}}-1}}
\ar[r]^{f_i^{(M/K)}} & {U_{\widetilde{\mathbb X}(M/K)}}
}
\end{equation}
is commutative. Thus,
\begin{equation*}
\begin{aligned}
\text{Pr}_{\widetilde{K_j\cap M}}\circ f_i^{(M/K)}(w)&=
\text{Pr}_{\widetilde{K_j\cap M}}\circ\widetilde{\mathcal N}_{L/M}
\left(f_i^{(L/K)}(v)\right)\\
&=\widetilde{N}_{K_j/K_j\cap M}\left(\text{Pr}_{\widetilde{K}_j}
\circ f_i^{(L/K)}(v)\right)\\
&=\widetilde{N}_{K_j/K_j\cap M}\left((g_{j-1}^{(L/K)}\circ\cdots
\circ g_i^{(L/K)})(v)\right)\\
&=\left(g_{j-1}^{(M/K)}\circ\cdots\circ g_i^{(M/K)}\right)
\left(\widetilde{N}_{K_i/K_i\cap M}(v)\right),
\end{aligned}
\end{equation*}
which is the desired equality.

Let, for each $0<i\in\mathbb Z$, 
\begin{equation*}
Z_i^{(M/K)}=\text{im}(f_i^{(M/K)}).
\end{equation*} 
Then, by Lemma \ref{infinite-product} or by Lemma 4 of \cite{fesenko-2001}, 
for $z^{(i)}\in Z_i^{(M/K)}$, the product $\prod_i z^{(i)}$ converges to 
an element in $U_{\widetilde{\mathbb X}(M/K)}^\diamond$. Let
\begin{equation*}
Z_{M/K}\left(\{K_i\cap M,f_i^{(M/K)}\}\right)=\left\{\prod_iz^{(i)}:z^{(i)}\in
Z_i^{(M/K)}\right\},
\end{equation*}
which is a topological subgroup of $U_{\widetilde{\mathbb X}(M/K)}^\diamond$.
Introduce the topological $Gal(M/K)$-submodule 
$Y_{M/K}\left(\{K_i\cap M,f_i^{(M/K)}\}\right)=Y_{M/K}$ of 
$U_{\widetilde{\mathbb X}(M/K)}^\diamond$ by
\begin{equation*}
Y_{M/K}=\left\{y\in U_{\widetilde{\mathbb X}(M/K)}: y^{1-\varphi}\in 
Z_{M/K}\left(\{K_i\cap M,f_i^{(M/K)}\}\right)\right\}.
\end{equation*}
\begin{lemma}
The norm map 
$\widetilde{\mathcal N}_{L/M}:\widetilde{\mathbb X}(L/K)^\times\rightarrow
\widetilde{\mathbb X}(M/K)^\times$ introduced by eq.s (\ref{coleman-norm})
and (\ref{coleman-norm-def}) further satisfies
\begin{itemize}
\item[(i)] 
$\widetilde{\mathcal N}_{L/M}\left(Z_{L/K}(\{K_i,f_i^{(L/K)}\})\right)
\subseteq Z_{M/K}\left(\{K_i\cap M,f_i^{(M/K)}\}\right)$;
\item[(ii)]
$\widetilde{\mathcal N}_{L/M}(Y_{L/K})\subseteq Y_{M/K}$.
\end{itemize}
\end{lemma}
\begin{proof}
Recall that, $\widetilde{\mathcal N}_{L/M}:\widetilde{\mathbb X}(L/K)^\times
\rightarrow\widetilde{\mathbb X}(M/K)^\times$ is a continuous mapping. 
\begin{itemize}
\item[(i)]
For any choice of $z^{(i)}\in Z_i^{(L/K)}$, the continuity of the 
multiplicative arrow 
$\widetilde{\mathcal N}_{L/M}:\widetilde{\mathbb X}(L/K)^\times\rightarrow
\widetilde{\mathbb X}(M/K)^\times$ yields
\begin{equation*}
\widetilde{\mathcal N}_{L/M}\left(\prod_iz^{(i)}\right)=\prod_i
\widetilde{\mathcal N}_{L/M}(z^{(i)}),
\end{equation*}
where $\widetilde{\mathcal N}_{L/M}(z^{(i)})\in Z_i^{(M/K)}$ by the 
commutative square (\ref{square-L/M}).
\item[(ii)]
Now let $y\in Y_{L/K}$. Then, 
$y^{1-\varphi}\in Z_{L/K}(\{K_i,f_i^{(L/K)}\})$. Thus,
$\widetilde{\mathcal N}_{L/M}(y^{1-\varphi})
=\widetilde{\mathcal N}_{L/M}(y)^{1-\varphi}\in Z_{M/K}
\left(\{K_i\cap M,f_i^{(M/K)}\}\right)$ by part (i). Now the result follows.
\end{itemize}
\end{proof}
Thus, the norm map 
$\widetilde{\mathcal N}_{L/M}:\widetilde{\mathbb X}(L/K)^\times\rightarrow
\widetilde{\mathbb X}(M/K)^\times$ of eq. (\ref{coleman-norm}) 
defined by eq. (\ref{coleman-norm-def}) induces a group homomorphism,
which will again be called the \textit{Coleman norm map from $L$ to $M$},
\begin{equation}
\label{coleman-norm-2}
\widetilde{\mathcal N}_{L/M}^{\text{Coleman}}:
U_{\widetilde{\mathbb X}(L/K)}^\diamond/Y_{L/K}
\rightarrow U_{\widetilde{\mathbb X}(M/K)}^\diamond/Y_{M/K}
\end{equation}
and defined by
\begin{equation}
\label{coleman-norm-2-def}
\widetilde{\mathcal N}_{L/M}^{\text{Coleman}}(\overline{U})=
\widetilde{\mathcal N}_{L/M}(U).Y_{M/K},
\end{equation}
for every $U\in U_{\widetilde{\mathbb X}(L/K)}^\diamond$, where
$\overline{U}$ denotes, as usual, the coset $U.Y_{L/K}$ in
$U_{\widetilde{\mathbb X}(L/K)}^\diamond/Y_{L/K}$.

Let
\begin{equation*}
\Phi_{M/K}^{(\varphi)}:\text{Gal}(M/K)\rightarrow 
U_{\widetilde{\mathbb X}{(M/K)}}^\diamond/Y_{M/K}
\end{equation*}
be the corresponding Fesenko reciprocity map defined for the extension
$M/K$, where $Y_{M/K}=Y_{M/K}\left(\{K_i\cap M,f_i^{(M/K)}\}\right)$. 
\begin{theorem}
\label{square1}
For the Galois sub-extension $M/K$ of $L/K$,
the square
\begin{equation}
\label{square-L/M-Y}
\SelectTips{cm}{}\xymatrix{
{\text{Gal}(L/K)}\ar[r]^-{\Phi_{L/K}^{(\varphi)}}\ar[d]_{\text{res}_M} & 
{U_{\widetilde{\mathbb X}{(L/K)}}^\diamond/Y_{L/K}}
\ar[d]^{\widetilde{\mathcal N}^{Coleman}_{L/M}} \\
{\text{Gal}(M/K)}\ar[r]^-{\Phi_{M/K}^{(\varphi)}} & 
{U_{\widetilde{\mathbb X}{(M/K)}}^\diamond/Y_{M/K}},
}
\end{equation}
where the right-vertical arrow
\begin{equation*}
\widetilde{\mathcal N}^{Coleman}_{L/M} : 
U_{\widetilde{\mathbb X}{(L/K)}}^\diamond/Y_{L/K}
\rightarrow U_{\widetilde{\mathbb X}{(M/K)}}^\diamond/Y_{M/K}
\end{equation*}
is the Coleman norm map from $L$ to $M$ defined by eq.s (\ref{coleman-norm-2}) 
and (\ref{coleman-norm-2-def}), is commutative.
\end{theorem}
\begin{proof}
It suffices to prove that the square
\begin{equation*}
\SelectTips{cm}{}\xymatrix{
{U_{\widetilde{\mathbb X}(L/K)}^\diamond/U_{\mathbb X(L/K)}}
\ar[r]^{\text{can.}}\ar[d]_{\widetilde{\mathcal N}^{Coleman}_{L/M}} & 
{U_{\widetilde{\mathbb X}(L/K)}^\diamond/Y_{L/K}}
\ar[d]^{\widetilde{\mathcal N}^{Coleman}_{L/M}} \\
{U_{\widetilde{\mathbb X}(M/K)}^\diamond/U_{\mathbb X(M/K)}}
\ar[r]^{\text{can.}} & 
{U_{\widetilde{\mathbb X}(M/K)}^\diamond/Y_{M/K}}
}
\end{equation*}
is commutative, which is obvious. Then pasting this square with the 
square eq. (\ref{square-L/M-U}) as
\begin{equation*}
\SelectTips{cm}{}\xymatrix{
{Gal(L/K)}\ar[r]^-{\phi_{L/K}^{(\varphi)}}\ar[d]_{\text{res}_M} & 
{U_{\widetilde{\mathbb X}(L/K)}^\diamond/U_{\mathbb X(L/K)}}
\ar[r]^{\text{can.}}\ar[d]^{\widetilde{\mathcal N}^{Coleman}_{L/M}} & 
{U_{\widetilde{\mathbb X}(L/K)}^\diamond/Y_{L/K}}
\ar[d]^{\widetilde{\mathcal N}^{Coleman}_{L/M}} \\
{Gal(M/K)}\ar[r]^-{\phi_{M/K}^{(\varphi)}} & 
{U_{\widetilde{\mathbb X}(M/K)}^\diamond/U_{\mathbb X(M/K)}}
\ar[r]^{\text{can.}} & 
{U_{\widetilde{\mathbb X}(M/K)}^\diamond/Y_{M/K}}
}
\end{equation*}
the commutativity of the square eq. (\ref{square-L/M-Y}) follows. 
\end{proof}
Now, let $F/K$ be a finite sub-extension of $L/K$. 
Then, as $F$ is \textit{compatible} with $(K,\varphi)$, in the 
sense of \cite{koch-deshalit}
pp. 89, we may fix the Lubin-Tate splitting over $F$ to be
$\varphi_F=\varphi_K=\varphi$. 
Thus, there exists the chain of field extensions 
\begin{equation*}
K\subseteq F\subseteq L\subseteq K_\varphi\subseteq F_{\varphi},
\end{equation*}
where $L$ is a totally-ramified $APF$-Galois extension over $F$ by
Lemma \ref{apftower}. As
$\pmb{\mu}_p(K^{sep})=\pmb{\mu}_p(F^{sep})$, the inclusion
$\pmb{\mu}_p(F^{sep})\subset F$ is satisfied. That is, the local field
$F$ satisfies the condition given by eq. (\ref{rootofunity}).

Now, the basic ascending chain of sub-extensions in $L/K$ fixed 
in eq. (\ref{basic-ascending-chain-L/K}) base changed to $F$ 
\begin{equation}
\label{basic-ascending-chain-L/F}
F=K_oF\subseteq K_1F\subseteq\cdots\subseteq K_iF
\subseteq\cdots\subseteq L
\end{equation}
is \textit{almost} a basic ascending chain of sub-extensions in $L/F$, which
follows by the isomorphisms
$\text{res}_{K_i} : \text{Gal}(K_iF/F)\simeq\text{Gal}(K_i/K_i\cap F)$ and 
$\text{res}_{K_i} : \text{Gal}(K_{i+1}F/K_iF)\simeq\text{Gal}(K_{i+1}/K_{i+1}
\cap K_iF)$ 
for every $0\leq i\in\mathbb Z$. Moreover, by primitive element theorem,
there exists an $0\leq i_o\in\mathbb Z$, such that $F\subseteq K_{i_o}$.
Choosing the minimal such $i_o$, the ascending chain 
(\ref{basic-ascending-chain-L/F}) becomes
\begin{equation*}
F=K_oF\subseteq K_1F\subseteq\cdots\subseteq K_{i_o-1}F=K_{i_o}
\subset K_{i_o+1}\subset\cdots\subset L .
\end{equation*}
For each $1\leq i\in\mathbb Z$, denote by $\sigma_i$ the element
in $\text{Gal}(\widetilde{L}/\widetilde{K})$ that satisfies
\begin{equation*}
<\sigma_i\mid_{K_i}>=\text{Gal}(K_i/K_{i-1}).
\end{equation*} 
Now, for $1\leq i\in\mathbb Z$, introduce the elements $\sigma_i^*$
in $\text{Gal}(\widetilde{L}/\widetilde{F})$ that satisfies
\begin{equation*} 
<\sigma_i^*\mid_{K_iF}>=\text{Gal}(K_iF/K_{i-1}F)
\end{equation*}
as follows :
\begin{itemize}
\item[(i)] in case $i>i_o$, then define $\sigma_i^*=\sigma_i$ ; 
\item[(ii)] in case $i\leq i_o$, then define 
\begin{equation*}
\sigma_i^*=
\begin{cases}
\sigma_i, & K_{i-1}F\subset K_{i}F ; \\
\text{id}_{K_iF}, & K_{i-1}F=K_{i}F .
\end{cases}
\end{equation*}
\end{itemize}
It is then clear that, for each $1\leq i\in\mathbb Z$, the elements 
$\sigma_i^*$ of $\text{Gal}(\widetilde{L}/\widetilde{F})$
satisfies
\begin{equation*}
<\sigma_i^*\mid_{K_iF}>=\text{Gal}(K_iF/K_{i-1}F),
\end{equation*}
and for almost all $i$, $\sigma_i^*=\sigma_i$.
Moreover, for each $1\leq k\in\mathbb Z$, the square
\begin{equation*}
\SelectTips{cm}{}\xymatrix{
{\text{Gal}(K_{k+1}F/K_kF)}\ar[r]^{\Xi_{K_{k+1}F/K_kF}}
\ar[d]_{r_{K_{k+1}}} & 
{U_{\widetilde{K_{k+1}F}}/U_{\widetilde{K_{k+1}F}}^{\sigma_{k+1}^*-1}}
\ar[d]^{\widetilde{N}_{K_{k+1}F/K_{k+1}}^*} \\
{\text{Gal}(K_{k+1}/K_k)}\ar[r]^{\Xi_{K_{k+1}/K_k}} & 
{U_{\widetilde{K}_{k+1}}/U_{\widetilde{K}_{k+1}}
^{\sigma_{k+1}-1}},
}
\end{equation*} 
is commutative, as
$\widetilde{N}_{K_{k+1}F/K_{k+1}}(\pi_{K_{k+1}F})=\pi_{K_{k+1}}$
by the \textit{norm coherence} of the Lubin-Tate labelling
$(\pi_{K'})_{\underset{[K':K]<\infty}{\underbrace{K\subseteq K'}}
\subset K_\varphi}$. Hence,
\begin{equation*}
\widetilde{N}_{K_{k+1}F/K_{k+1}}^*\left(T_k^{(L/F)}\right)=T_k^{(L/K)}.
\end{equation*}
Now, by Theorem \ref{split-exact-sequence} there exists an arrow
\begin{equation*}
h_k^{(L/F)}:\prod_{1\leq i\leq k}
U_{\widetilde{K_kF}}^{\sigma_i^*-1}\rightarrow
\left(\prod_{1\leq i\leq k+1}U_{\widetilde{K_{k+1}F}}
^{\sigma_i^*-1}\right)/U_{\widetilde{K_{k+1}F}}
^{\sigma_{k+1}^*-1}
\end{equation*}
which splits the exact sequence
\begin{equation}
\label{split-short-exact-L/F}
\SelectTips{cm}{}\xymatrix@1{
1\ar[r] & T_k^{(L/F)'}\ar[r] & 
{\left(\prod_{1\leq i\leq k+1}U_{\widetilde{K_{k+1}F}}
^{\sigma_i^*-1}\right)/U_{\widetilde{K_{k+1}F}}^{\sigma_{k+1}^*-1}}
\ar[r]^-{\widetilde{N}_{K_{k+1}F/K_kF}} &
{\prod_{1\leq i\leq k}U_{\widetilde{K_kF}}
^{\sigma_i^*-1}}
\ar[r]\ar@/_2.2pc/[l]_{h_k^{(L/F)}}&1} .
\end{equation}
Now, choose an arrow
\begin{equation}
\label{h-L/K}
h_k^{(L/K)}:\prod_{1\leq i\leq k}
U_{\widetilde{K_k}}^{\sigma_i-1}\rightarrow
\left(\prod_{1\leq i\leq k+1}U_{\widetilde{K_{k+1}}}
^{\sigma_i-1}\right)/U_{\widetilde{K_{k+1}}}
^{\sigma_{k+1}-1}
\end{equation}
which splits the exact sequence (\ref{split-short-exact-L/K})
in such a way that
\begin{equation}
\label{diagram-h-L/F}
\SelectTips{cm}{}\xymatrix{
{\left(\prod_{1\leq i\leq k+1}U_{\widetilde{K_{k+1}F}}
^{\sigma_i^*-1}\right)/U_{\widetilde{K_{k+1}F}}^{\sigma_{k+1}^*-1}}
\ar[r]^-{\widetilde{N}_{K_{k+1}F/K_kF}}
\ar[d]_{\widetilde{N}^*_{K_{k+1}F/K_{k+1}}} & 
{\prod_{1\leq i\leq k}U_{\widetilde{K_kF}}^{\sigma_i^*-1}}
\ar@/_2.2pc/[l]_{h_k^{(L/F)}}
\ar[d]^{\widetilde{N}_{K_kF/K_k}} \\
{\left(\prod_{1\leq i\leq k+1}U_{\widetilde{K}_{k+1}}
^{\sigma_i-1}\right)
/U_{\widetilde{K}_{k+1}}^{\sigma_{k+1}-1}}
\ar[r]^-{\widetilde{N}_{K_{k+1}/K_k}} & 
{\prod_{1\leq i\leq k}U_{\widetilde{K}_k}
^{\sigma_i-1}}
\ar@/^2.2pc/[l]^{h_k^{(L/K)}}
}
\end{equation}
is a commutative square. The arrow eq. (\ref{h-L/K}) is constructed by
following the same lines of the construction of the arrow eq. (\ref{h-M/K}).
For each $1\leq k\in\mathbb Z$, consider any map
\begin{equation}
\label{g-L/K}
g_k^{(L/K)}:\prod_{1\leq i\leq k}U_{\widetilde{K}_k}
^{\sigma_i-1}\rightarrow\prod_{1\leq i\leq k+1}
U_{\widetilde{K}_{k+1}}^{\sigma_i-1}
\end{equation}
which commutes the following square
\begin{equation}
\label{diagram-g-L/F}
\SelectTips{cm}{}\xymatrix{
{\prod_{1\leq i\leq k}U_{\widetilde{K_kF}}^{\sigma_i^*-1}}
\ar[r]^-{g_k^{(L/F)}}\ar[d]_{\widetilde{N}_{K_kF/K_k}} & 
{\prod_{1\leq i\leq k+1}U_{\widetilde{K_{k+1}F}}^{\sigma_i^*-1}}
\ar[d]^{\widetilde{N}_{K_{k+1}F/K_{k+1}}} \\
{\prod_{1\leq i\leq k}U_{\widetilde{K}_k}
^{\sigma_i-1}}\ar@{.>}[r]^-{g_k^{(L/K)}} & 
{\prod_{1\leq i\leq k+1}U_{\widetilde{K}_{k+1}}
^{\sigma_i-1}}.
}
\end{equation}
Note that, such a map satisfies
\begin{equation*}
h_k^{(L/K)}=g_k^{(L/K)}\mod{U_{\widetilde{K}_{k+1}}^{\sigma_{k+1}-1}}.
\end{equation*}
In fact, by the commutative diagram (\ref{diagram-h-L/F}),
for any $w\in\prod_{1\leq i\leq k}U_{\widetilde{K}_k}^{\sigma_i-1}$, 
there exists 
$v\in\prod_{1\leq i\leq k}U_{\widetilde{K_kF}}^{\sigma_i^*-1}$ such that
$w=\widetilde{N}_{K_kF/K_k}(v)$, and
\begin{equation*}
\begin{aligned}
h_k^{(L/K)}(w)&=h_k^{(L/K)}\left(\widetilde{N}_{K_kF/K_k}(v)\right)\\
&=\widetilde{N}^*_{K_{k+1}F/K_{k+1}}\left(h_k^{(L/F)}(v)\right)\\
&=\widetilde{N}^*_{K_{k+1}F/K_{k+1}}\left(g_k^{(L/F)}(v).
U_{\widetilde{K_{k+1}F}}^{\sigma_{k+1}^*-1}\right)\\
&=\widetilde{N}_{K_{k+1}F/K_{k+1}}\left(g_k^{(L/F)}(v)\right).
U_{\widetilde{K}_{k+1}}^{\sigma_{k+1}-1},
\end{aligned}
\end{equation*}
and by the commutativity of the diagram (\ref{diagram-g-L/F}),
\begin{equation*}
\begin{aligned}
\widetilde{N}_{K_{k+1}F/K_{k+1}}\left(g_k^{(L/F)}(v)\right)&=
g_k^{(L/K)}\left(\widetilde{N}_{K_kF/K_k}(v)\right)\\
&=g_k^{(L/K)}(w).
\end{aligned}
\end{equation*}
Thus, the equality
\begin{equation*}
h_k^{(L/K)}(w)=g_k^{(L/K)}(w).U_{\widetilde{K}_{k+1}}^{\sigma_{k+1}-1}
\end{equation*}
follows for every 
$w\in\prod_{1\leq i\leq k}U_{\widetilde{K}_k}^{\sigma_i-1}$.

Now, for each $1\leq i\in\mathbb Z$, introduce the map
\begin{equation*}
f_i^{(L/K)}: U_{\widetilde{K}_i}^{\sigma_i-1}\rightarrow 
U_{\widetilde{\mathbb X}(L/K)}
\end{equation*}
by
\begin{equation*}
f_i^{(L/K)}(w)=\Lambda_{F/K}\left(f_i^{(L/F)}(v)\right),
\end{equation*}
where $v\in U_{\widetilde{K_iF}}^{\sigma_i^*-1}$ is any element satisfying
$\widetilde{N}_{K_iF/K_i}(v)=w\in U_{\widetilde{K}_i}^{\sigma_i-1}$. 
Note that, if 
$v'\in U_{\widetilde{K_iF}}^{\sigma_i^*-1}$ such that
$\widetilde{N}_{K_iF/K_i}(v')=w$, then
$\Lambda_{F/K}\left(f_i^{(L/F)}(v)\right)=\Lambda_{F/K}
\left(f_i^{(L/F)}(v')\right)$. 
In fact, there 
exists $u\in\ker\left(\widetilde{N}_{K_iF/K_i}\right)$ such that
$v'=vu$. Thus, we have to verify that
$\Lambda_{F/K}\left(f_i^{(L/F)}(v)\right)=
\Lambda_{F/K}\left(f_i^{(L/F)}(vu)\right)$. That is, for each 
$1\leq j\in\mathbb Z$, we have to check the equality
\begin{equation}
\label{f-L/K-well-defined}
\text{Pr}_{\widetilde{K}_j}\left(\Lambda_{F/K}\left(f_i^{(L/F)}(v)
\right)\right)=
\text{Pr}_{\widetilde{K}_j}\left(\Lambda_{F/K}\left(f_i^{(L/F)}(vu)
\right)\right).
\end{equation}
Now, note that, for $j>i$,
\begin{equation*}
\begin{aligned}
\text{Pr}_{\widetilde{K}_j}\left(\Lambda_{F/K}\left(f_i^{(L/F)}(v)
\right)\right)&=
\widetilde{N}_{K_jF/K_j}\left(\text{Pr}_{\widetilde{K_jF}}
\left(\Lambda_{F/K}(f_i^{(L/F)}(v))\right)\right) \\
&=\widetilde{N}_{K_jF/K_j}\left(\text{Pr}_{\widetilde{K_jF}}
\left(f_i^{(L/F)}(v)\right)\right)\\
&=\widetilde{N}_{K_jF/K_j}\left(g_{j-1}^{(L/F)}\circ\cdots
\circ g_{i}^{(L/F)}(v)\right)\\
&=g_{j-1}^{(L/K)}\circ\cdots\circ g_{i}^{(L/K)}
\left(\widetilde{N}_{K_iF/K_i}(v)\right),
\end{aligned}
\end{equation*}
by the commutativity of the square eq. (\ref{diagram-g-L/F}). Thus, equality
(\ref{f-L/K-well-defined}) follows, as $\widetilde{N}_{K_iF/K_i}(v)
=\widetilde{N}_{K_iF/K_i}(vu)$. Thus, the map 
\begin{equation*}
f_i^{(L/K)}: U_{\widetilde{K}_i}^{\sigma_i-1}\rightarrow 
U_{\widetilde{\mathbb X}(L/K)}
\end{equation*}
is well-defined. Moreover, for $j>i$,
\begin{equation*}
\text{Pr}_{\widetilde{K}_j}\circ f_i^{(L/K)}=\left(g_{j-1}^{(L/K)}\circ
\cdots\circ g_i^{(L/K)}\right)\mid_{U_{\widetilde{K}_i}^{\sigma_i-1}}.
\end{equation*}
In fact, for $w\in U_{\widetilde{K}_i}^{\sigma_i-1}$,
there exists $v\in U_{\widetilde{K_iF}}^{\sigma_i^*-1}$ such that
$\widetilde{N}_{K_iF/K_i}(v)=w$, and 
$f_i^{(L/K)}(w)=\Lambda_{F/K}\left(f_i^{(L/F)}(v)\right)$. 
That is, the following square
\begin{equation}
\label{square-F/K}
\SelectTips{cm}{}\xymatrix{
{U_{\widetilde{K_iF}}^{\sigma_i^*-1}}\ar[r]^{f_i^{(L/F)}}
\ar[d]_{\widetilde{N}_{K_iF/K_i}} & 
{U_{\widetilde{\mathbb X}(L/F)}}\ar[d]^{\Lambda_{F/K}} \\
{U_{\widetilde{K}_i}^{\sigma_i-1}}
\ar[r]^{f_i^{(L/K)}} & {U_{\widetilde{\mathbb X}(L/K)}}
}
\end{equation}
is commutative. Thus, for $j>i$,
\begin{equation*}
\begin{aligned}
\text{Pr}_{\widetilde{K}_j}\circ f_i^{(L/K)}(w)&=
\text{Pr}_{\widetilde{K}_j}\circ\Lambda_{F/K}
\left(f_i^{(L/F)}(v)\right)\\
&=\widetilde{N}_{K_jF/K_j}\left(\text{Pr}_{\widetilde{K_jF}}
\circ f_i^{(L/F)}(v)\right)\\
&=\widetilde{N}_{K_jF/K_j}\left((g_{j-1}^{(L/F)}\circ\cdots
\circ g_i^{(L/F)})(v)\right)\\
&=\left(g_{j-1}^{(L/K)}\circ\cdots\circ g_i^{(L/K)}\right)
\left(\widetilde{N}_{K_iF/K_i}(v)\right),
\end{aligned}
\end{equation*}
by the commutativity of the diagram (\ref{diagram-g-L/F}), which is the 
desired equality.

Let, for each $0<i\in\mathbb Z$, 
\begin{equation*}
Z_i^{(L/K)}=\text{im}(f_i^{(L/K)}).
\end{equation*} 
Then, by Lemma \ref{infinite-product} or by Lemma 4 of \cite{fesenko-2001}, 
for $z^{(i)}\in Z_i^{(L/K)}$, the product $\prod_i z^{(i)}$ converges to 
an element in $U_{\widetilde{\mathbb X}(L/K)}^\diamond$. Let
\begin{equation*}
Z_{L/K}\left(\{K_i,f_i^{(L/K)}\}\right)=\left\{\prod_iz^{(i)}:z^{(i)}\in
Z_i^{(L/K)}\right\},
\end{equation*}
which is a topological subgroup of $U_{\widetilde{\mathbb X}(L/K)}^\diamond$.
Introduce the topological $Gal(L/K)$-submodule 
$Y_{L/K}\left(\{K_i,f_i^{(L/K)}\}\right)=Y_{L/K}$ of 
$U_{\widetilde{\mathbb X}(L/K)}^\diamond$ by
\begin{equation*}
Y_{L/K}=\left\{y\in U_{\widetilde{\mathbb X}(L/K)}: y^{1-\varphi}\in 
Z_{L/K}\left(\{K_i,f_i^{(L/K)}\}\right)\right\}.
\end{equation*}
\begin{lemma}
The continuous homomorphism
$\Lambda_{F/K}:\widetilde{\mathbb X}(L/F)^\times\rightarrow
\widetilde{\mathbb X}(L/K)^\times$ introduced by eq.s (\ref{Lambda-map})
and (\ref{Lambda-map-definition}) further satisfies
\begin{itemize}
\item[(i)] 
$\Lambda_{F/K}\left(Z_{L/F}(\{K_iF,f_i^{(L/F)}\})\right)
\subseteq Z_{L/K}\left(\{K_i,f_i^{(L/K)}\}\right)$;
\item[(ii)]
$\Lambda_{F/K}(Y_{L/F})\subseteq Y_{L/K}$.
\end{itemize}
\end{lemma}
\begin{proof} 
\begin{itemize}
\item[(i)]
For any choice of $z^{(i)}\in Z_i^{(L/F)}$, the continuity of the 
multiplicative arrow 
$\Lambda_{F/K}:\widetilde{\mathbb X}(L/F)^\times\rightarrow
\widetilde{\mathbb X}(L/K)^\times$ yields
\begin{equation*}
\Lambda_{F/K}\left(\prod_iz^{(i)}\right)=\prod_i\Lambda_{F/K}(z^{(i)}),
\end{equation*}
where $\Lambda_{F/K}(z^{(i)})\in Z_i^{(L/K)}$ by the 
commutative square (\ref{square-F/K}).
\item[(ii)]
Now let $y\in Y_{L/F}$. Then, 
$y^{1-\varphi}\in Z_{L/F}(\{K_iF,f_i^{(L/F)}\})$. Thus,
$\Lambda_{F/K}(y^{1-\varphi})
=\Lambda_{F/K}(y)^{1-\varphi}\in Z_{L/K}
\left(\{K_i,f_i^{(L/K)}\}\right)$ by part (i). Now the result follows.
\end{itemize}
\end{proof}
Thus, the homomorphism
$\Lambda_{F/K}:\widetilde{\mathbb X}(L/F)^\times\rightarrow
\widetilde{\mathbb X}(L/K)^\times$ of eq. (\ref{Lambda-map}) 
defined by eq. (\ref{Lambda-map-definition}) induces a group homomorphism,
\begin{equation}
\label{lambda-map-2}
\lambda_{F/K}:
U_{\widetilde{\mathbb X}(L/F)}^\diamond/Y_{L/F}
\rightarrow U_{\widetilde{\mathbb X}(L/K)}^\diamond/Y_{L/K}
\end{equation}
and defined by
\begin{equation}
\label{lambda-map-definition-2}
\lambda_{F/K}(\overline{U})=
\Lambda_{F/K}(U).Y_{L/K},
\end{equation}
for every $U\in U_{\widetilde{\mathbb X}(L/F)}^\diamond$, where
$\overline{U}$ denotes, as usual, the coset $U.Y_{L/F}$ in
$U_{\widetilde{\mathbb X}(L/F)}^\diamond/Y_{L/F}$.

Let
\begin{equation*}
\Phi_{L/F}^{(\varphi)}:\text{Gal}(L/F)\rightarrow
U_{\widetilde{\mathbb X}{(L/F)}}^\diamond/Y_{L/F}
\end{equation*}
be the corresponding Fesenko reciprocity map defined for the extension 
$L/F$, where $Y_{L/F}=Y_{L/F}\left(\{K_iF,f_i^{(L/F)}\}\right)$.
\begin{theorem}
\label{square2}
For the finite sub-extension $F/K$ of $L/K$, the square
\begin{equation}
\label{square-F/K-Y}
\SelectTips{cm}{}\xymatrix{
{\text{Gal}(L/F)}\ar[r]^-{\Phi_{L/F}^{(\varphi)}}\ar[d]_{\text{inc.}} & 
{U_{\widetilde{\mathbb X}{(L/F)}}^\diamond/Y_{L/F}}\ar[d]^{\lambda_{F/K}} \\
{\text{Gal}(L/K)}\ar[r]^-{\Phi_{L/K}^{(\varphi)}} & 
{U_{\widetilde{\mathbb X}{(L/K)}}^\diamond/Y_{L/K}},
}
\end{equation}
where the right-vertical arrow
\begin{equation*}
\lambda_{F/K}:U_{\widetilde{\mathbb X}{(L/F)}}^\diamond/Y_{L/F}
\rightarrow U_{\widetilde{\mathbb X}{(L/K)}}^\diamond/Y_{L/K}
\end{equation*}
is defined by eq.s (\ref{lambda-map-2}) and (\ref{lambda-map-definition-2}), 
is commutative.
\end{theorem}
\begin{proof}
It suffices to prove that the square
\begin{equation*}
\SelectTips{cm}{}\xymatrix{
{U_{\widetilde{\mathbb X}(L/F)}^\diamond/U_{\mathbb X(L/F)}}
\ar[r]^{\text{can.}}\ar[d]_{\lambda_{F/K}} & 
{U_{\widetilde{\mathbb X}(L/F)}^\diamond/Y_{L/F}}
\ar[d]^{\lambda_{F/K}} \\
{U_{\widetilde{\mathbb X}(L/K)}^\diamond/U_{\mathbb X(L/K)}}
\ar[r]^{\text{can.}} & 
{U_{\widetilde{\mathbb X}(L/K)}^\diamond/Y_{L/K}}
}
\end{equation*}
is commutative, which is obvious. Then pasting this square with the 
square eq. (\ref{square-F/K-U}) as
\begin{equation*}
\SelectTips{cm}{}\xymatrix{
{Gal(L/F)}\ar[r]^-{\phi_{L/F}^{(\varphi)}}\ar[d]_{\text{inc.}} & 
{U_{\widetilde{\mathbb X}(L/F)}^\diamond/U_{\mathbb X(L/F)}}
\ar[r]^{\text{can.}}\ar[d]^{\lambda_{L/F}} & 
{U_{\widetilde{\mathbb X}(L/F)}^\diamond/Y_{L/F}}
\ar[d]^{\lambda_{L/F}} \\
{Gal(L/K)}\ar[r]^-{\phi_{L/K}^{(\varphi)}} & 
{U_{\widetilde{\mathbb X}(L/K)}^\diamond/U_{\mathbb X(L/K)}}
\ar[r]^{\text{can.}} & 
{U_{\widetilde{\mathbb X}(L/K)}^\diamond/Y_{L/K}}
}
\end{equation*}
the commutativity of the square eq. (\ref{square-F/K-Y}) follows. 
\end{proof}
If $L/K$ is furthermore a finite extension, then the square
\begin{equation*}
\SelectTips{cm}{}\xymatrix{
{U_{\widetilde{\mathbb X}{(L/K)}}^\diamond/Y_{L/K}}
\ar[r]^-{H_{L/K}^{(\varphi)}}\ar[d]_{\text{Pr}_{\widetilde{K}}} & 
{\text{Gal}(L/K)}\ar[d]^{\txt{mod \text{Gal}(L/K)}'} \\
{U_K/N_{L/K}U_L}\ar[r]^-{h_{L/K}} & {\text{Gal}(L/K)^{ab}}
}
\end{equation*}
commutes. Thus \textit{the inverse $H_{L/K}^{(\varphi)}
=(\Phi_{L/K}^{(\varphi)})^{-1}$ of the 
Fesenko reciprocity map $\Phi_{L/K}^{(\varphi)}$ defined for $L/K$ 
is the generalization of the Hazewinkel map for the totally-ramified 
$APF$-Galois sub-extensions $L/K$ of $K_\varphi/K$ under the
assumption that the local field $K$ satisfies the condition given 
by eq. (\ref{rootofunity})}.

${}$

\noindent
\textsc{Department of Mathematics, Istanbul Bilgi University, Kurtulu\c s 
Deresi Cad. No. 47, Dolapdere, 34440 Beyo\v glu, Istanbul, TURKEY}

\noindent
E-mail : \texttt{ilhan$@$bilgi.edu.tr}

\noindent
\textsc{Department of Mathematics, Atilim University, Kizilca\c sar
K\"oy\"u, Incek, 06836 G\"olba\c s\i, Ankara, TURKEY}

\noindent
E-mail : \texttt{erols73$@$yahoo.com}

\end{document}